\documentclass[12pt]{article}

\usepackage[ddmmyyyy]{datetime}
\usepackage{geometry}
\usepackage{lineno,hyperref}
\usepackage{verbatim}
\nolinenumbers
\usepackage{float}
\usepackage{amsfonts}
\usepackage{mathrsfs}
\usepackage{amssymb}
\usepackage{amsmath}
\usepackage{amsthm}
\usepackage{appendix}
\usepackage{xcolor}
\usepackage[textfont=scriptsize,labelfont=bf]{caption}
\usepackage[textfont=scriptsize,labelfont=scriptsize]{subcaption}
\usepackage{multirow}
\usepackage{booktabs}
\usepackage{graphicx}
\usepackage{algpseudocode}

\usepackage{sparse}
\usepackage{sparsemicros}

\title{Sparse Isotropic Regularization for Spherical Harmonic\\
 Representations of Random Fields on the Sphere}

\author{Q.~T.~Le Gia, I.~H.~Sloan, R.~S.~Womersley \\ 
School of Mathematics and Statistics\\
The University of New South Wales, Sydney, NSW 2052\\
Australia
\and Yu Guang Wang\\
Department of Mathematics and Statistics\\
La Trobe University, Bundoora 3086 VIC\\ 
Australia.}
\date{}

\begin{document}
\maketitle
\begin{abstract}
This paper discusses sparse isotropic regularization for a random field on
the unit sphere $\sph{2}$ in $\mathbb{R}^{3}$, where the field is expanded
in terms of a spherical harmonic basis. A key feature is that the norm used in
the regularization term, a hybrid of the $\ell_{1}$ and $\ell_2$-norms, is chosen so that the
regularization preserves isotropy, in the sense that if the observed
random field is strongly isotropic then so too is the regularized field.
The Pareto efficient frontier is used to display the trade-off between the sparsity-inducing norm and the data discrepancy term, in order to help in the choice of a suitable regularization parameter.
 A numerical example using Cosmic Microwave Background (CMB) data is considered in detail.
 In particular, the numerical results explore the trade-off between regularization and discrepancy,
and show that substantial sparsity can be achieved along with small $L_{2}$ error.
\end{abstract}


\section{Introduction}

This paper presents a new algorithm for the sparse regularization of a real-valued random field $T$ on the sphere, with the regularizer taken to be a novel norm (a hybrid of $\ell_1$ and $\ell_2$ norms) imposed on the coefficients $\coesh$ of the spherical harmonic decomposition,
\begin{equation*}
\RF(\PT{x})=\sum_{\ell=0}^\infty\sum_{m=-\ell}^\ell \coesh \shY(\PT{x}),
\quad \PT{x}\in \mathbb{S}^2.
\end{equation*}
Here $\shY$ for $m=-\ell,\ldots,\ell$ is a (complex) orthonormal
basis for the space of homogeneous harmonic polynomials of degree $\ell$ in $\mathbb{R}^3$, restricted to the unit sphere
$\mathbb{S}^2:=\{\PT{x}\in\mathbb{R}^3:|\PT{x}|=1\}$, with $|\cdot|$
denoting the Euclidean norm in $\mathbb{R}^3$.

Random fields on the sphere have recently attracted much attention from both mathematicians \cite{MaPe2011} and astrophysicists. In particular, the satellite data used to form the map of the Cosmic Microwave Background (see \cite{Planck2016I,Planck2016IX,Planck2016XVI}), is usually viewed as, to a good approximation, a single realization of an isotropic Gaussian random field, after correction for the obscured portion of the map near the galactic plane.

Sparse regularization of data (i.e. a regularized approximation in which many coefficients in an expansion are zero) is another topic that has recently attracted great attention, especially in compressed sensing and signal analysis, see for example
\cite{CanRT2006a,DauFL2008,Don2006}. In the context of CMB the use
of sparse representations is somewhat controversial, see for example
\cite{StDoFaRa2013}, but nevertheless has often been discussed, especially in the context of inpainting to correct for the obscuring effect of our galaxy near the galactic plane.

In a recent paper, Cammarota and Marinucci \cite{CamMar15} considered a particular $\ell_1$-regularization problem based on spherical harmonics, and showed that if the true field is both Gaussian and isotropic (the latter meaning that the underlying law is invariant under rotation), then the resulting regularized solution is neither Gaussian nor isotropic. The problem of anisotropy has also been pointed out in sparse inpainting on the sphere \cite{Fe_etal2014}.

The scheme analyzed in \cite{CamMar15} obtains a regularized field as the minimizer of
\begin{equation}\label{eq:CRminimiser}
\frac{1}{2}\|\RF-\RFo\|_{L_2(\mathbb{S}^2)}^2+
\lambda\sum_{\ell=0}^\infty\sum_{m=-\ell}^\ell |\coesh|,
\end{equation}
where $\RFo$ is the observed field, and $\lambda\ge0$ is a regularization parameter. Behind the non-preservation of isotropy in this scheme lies a more fundamental problem, namely that the regularizer in \eqref{eq:CRminimiser} is not invariant under rotation of the coordinate axes. For this reason the regularized field, and even the sparsity pattern, will in general depend on the choice of coordinate axes.

The essential point is that for a given $\ell\ge 1$ the sum
$\sum_{m=-\ell}^\ell|\coesh|^2$ is rotationally invariant, while the sum $\sum_{m=-\ell}^\ell|\coesh|$ is not. For convenience the rotational invariance property is proved in the next section.

A simple example might be illuminating.  Suppose that a particular
realization of the field happens to take the (improbable!) form
$$
\RF(\PT{x})=\PT{x}\cdot\PT{p}, \quad \PT{x}\in \mathbb{S}^2
$$
for some fixed point $\PT{p}$ on the celestial unit sphere.  If the $z$ axis is chosen so that $\PT{p}$ is at the north pole, then $T(\PT{x})=\cos\theta$ where $\theta$ is the usual polar angle, and so $\RF(\PT{x})=\alpha\shY[1,0](\PT{x})$, where $\alpha=\sqrt{4\pi/3}$ (since $\shY[1,0](\theta,\phi) = \sqrt{3/(4\pi)} \cos\theta$). Thus with this choice we have $\coesh[1,0]= \alpha$, and all other coefficients are zero. On the other
hand, if the axes are chosen so that $\PT{p}$ lies on the $x$ axis then the field has the polar coordinate representation
$$
\RF(\PT{x})=\sin\theta \cos \phi
=\frac{1}{2} (e^{i\phi} +e^{-i\phi})\sin \theta
= \frac{\alpha}{\sqrt{2}}(\shY[1,1]-\shY[1,-1]),
$$
so that now the only non-zero coefficients are $\coesh[1,1]=-\coesh[1,-1]=\alpha/\sqrt{2}$. Note that the sum of the absolute values in the second case is larger than that in the first case by a factor of $\sqrt{2}$.  (Note that even the choice of complex basis for the spherical harmonics affects the sum of the absolute values of the coefficients; but not, of course, the sum of the squares of the absolute values.)

With this motivation, in this paper we replace the regularizer in
\eqref{eq:CRminimiser} by one that is manifestly rotationally invariant:
in our scheme the regularized field is the minimizer of
\begin{equation}
\label{eq:our_minimiser}
\frac{1}{2}\norm{\RF-\RFo}{\Lp{2}{2}}^2+
\lambda\sum_{\ell=0}^\infty\beta_\ell \bigg(\sum_{m=-\ell}^\ell |\coesh|^2\bigg)^{1/2},
\end{equation}
where the $\beta_\ell$ are at this point arbitrary positive numbers
normalized by $\beta_0 =1$.
With an appropriate choice of $(\beta_\ell)_{\ell\in\Nz}$ and $\lambda$
our regularized solution will turn out to be sparse, but with
the additional property of either preserving all or discarding all the coefficients $\coesh$ of a given degree $\ell$.

It is easily seen that the regularized field, that is the minimizer of \eqref{eq:our_minimiser} for a given observed field $\RFo$, takes the form
\begin{equation*}
	\RFr(\PT{x}) = \sum_{\ell=0}^{\infty}\sum_{m=-\ell}^{\ell}\coeshr \shY(\PT{x}),\quad \PT{x}\in\sph{2},
\end{equation*}
where (see Proposition~\ref{prop:reg.sol})
\begin{equation*}
   \coeshr := \left\{\begin{array}{ll}
	\displaystyle \left(1 - \frac{\lambda\beta_{\ell}}{\Alo}\right) \coesho, & \mbox{if~}\Alo>\lambda\beta_{\ell},\\
	0, & \mbox{if~}\Alo\le\lambda\beta_{\ell},
 \end{array}\right.
\end{equation*}
where
\begin{equation}\label{eq:Alo}
  \Alo := \left(\sum_{m=-\ell}^\ell | \coesho |^2 \right)^\frac{1}{2},
  \quad \ell\ge0.
\end{equation}

Since the resulting sparsity pattern depends entirely on the sequence of ratios $\Alo/(\lambda\beta_{\ell})$ for $\ell\ge0$, it is clear that in any application of the present
regularization scheme, the choices of the sequence $(\beta_{\ell})_{\ell\ge0}$
and the parameter $\lambda$ are crucial. In this paper we shall discuss these choices
in relation to a particular dataset from the cosmic microwave background (CMB) project,
first choosing $\beta_{\ell}$
to match the observed decay of the $\Alo$, and finally choosing the parameter $\lambda$.
We shall see that the resulting sparsity can vary greatly as $\lambda$ varies for
given $(\beta_{\ell})_{\ell\in\Nz}$, with little
change to the $L_{2}$ error of the approximation.

Because of its very nature, the regularized solution has in general
a smaller norm than the observed field. We therefore explore the option of scaling the
regularized field so that both the observed and regularized fields have the same $L_{2}$ norm.

The paper is organized as follows: In Section~\ref{sec:prelim} we review key definitions and
properties of isotropic random fields on the unit sphere, the choice of norm and the regularization model. In Section~\ref{sec:regsol} we give the analytic solution to the regularization model.
Section~\ref{sec:iso} proves that the regularization scheme produces a strongly isotropic field when the observed field is strongly isotropic.
Section~\ref{sec:err} estimates the approximation error of the sparsely regularized random field from the observed random field, and for a given error provides an upper estimate of the regularization parameter $\lambda$.
In Section~\ref{sec:scaling}, we consider the option of scaling the regularized field
so that the $L_2$-norm is preserved.
In Section~\ref{sec:num} we describe the numerical experiments that illustrate the proposed
regularization algorithm. In particular,
Section~\ref{sec:beta} considers the choice of the scaling parameters in the norm, while
Section~\ref{sec:eff} illustrates use of the Pareto efficient frontier to help guide the choice of regularization parameter. Finally,
Section~\ref{sec:CMB} uses the CMB data to illustrate the regularization scheme.


\section{Preliminaries}\label{sec:prelim}

\subsection{Rotational invariance}
In this subsection, randomness plays no role.
Let $\sph{2}$ be the unit sphere in the Euclidean space $\R^3$. Let $\Lp{2}{2}:=L_{2}(\sph{2},\sigma)$ denote the space
of complex-valued square integrable functions on $\sph{2}$ with the surface measure $\sigma$ on $\sph{2}$ satisfying $\sigma(\sph{2})=4\pi$, endowed with the inner product
$\int_{\sph{2}}f(\PT{x})\conj{g(\PT{x})}\IntDiff[]{x}$ for $f,g\in\Lp{2}{2}$ and with
induced $L_2$-norm $\norm{f}{\Lp{2}{2}}=\sqrt{\int_{\sph{2}}|f(\PT{x})|^2\IntDiff[]{x}}$.
The (complex-valued) spherical
harmonics $\{\shY : \ell=0,1,2,\ldots; m=-\ell,\ldots,\ell\}$, which are the eigenfunctions of the Laplace-Beltrami operator for the sphere, form a complete orthonormal basis for $\Lp{2}{2}$.
There are various spherical harmonic definitions. This paper uses the basis as in \cite{Liboff2003}, which is widely used in physics.

A function $f\in \Lp{2}{2}$
can be expanded in terms of a Fourier-Laplace series
\begin{equation}\label{eq:Fseries}
   f \sim \sum_{\ell=0}^\infty \sum_{m=-\ell}^{\ell} \Fcoe{f} \shY, \text{ with }
            \widehat{f}_{\ell,m} := \int_{\sph{2}} f(\PT{x}) \conj{\shY(\PT{x})}
            \IntDiff[]{x},
\end{equation}
with $\sim$ denoting the convergence in the $\Lp{2}{2}$ sense. The $\Fcoe{f}$ are called
the \emph{Fourier coefficients} for $f$ under the Fourier basis $\shY$.
Parseval's theorem states that
\[
 \norm{f}{\Lp{2}{2}}^2 = \sum_{\ell=0}^\infty \sum_{m=-\ell}^{\ell} |\Fcoe{f}|^2, \qquad f \in \Lp{2}{2}.
\]

As promised in the Introduction, we now show that the sum over $m$ of the squared absolute values of the Fourier coefficients is rotationally invariant.

Let $\RotGr[3]$ be the rotation group on $\R^3$. For a
given rotation $\rho \in \RotGr[3]$ and a given function $f\in
L^2(\sph{2})$, the linear operator $\rtop$ on $\Lp{2}{2}$ associated with
the rotation $\rho$ is defined by
\begin{equation*}
  \rtop f(\PT{x}):=f(\rho^{-1}\PT{x}), \qquad \PT{x} \in \sph{2}.
\end{equation*}
The rotated function $\rtop f$ is essentially the same function as $f$,
but expressed with respect to a coordinate system rotated by $\rho$. The operators $\rtop$ form a representation of the group $\RotGr[3]$, in that
\begin{align*}
	(\rtop[\rho_{1}]\rtop[\rho_{2}])f(\PT{x})
	&= \rtop[\rho_{1}](\rtop[\rho_{2}]f)(\PT{x}) = (\rtop[\rho_{2}]f)(\rho_{1}^{-1}\PT{x}) = f(\rho_{2}^{-1}\rho_{1}^{-1}\PT{x})\\
&=f((\rho_{1}\rho_{2})^{-1}\PT{x})=\rtop[\rho_{1}\rho_{2}]f(\PT{x}), \quad \rho_{1},\rho_{2}\in \RotGr[3].
\end{align*}

\noindent \textbf{Definition} A (non-linear) functional $\funct$ of
$f\in L_2(\sph{2})$ is rotationally invariant if for all rotations
$\rho\in \RotGr[3]$,
\[
	\funct(\rtop f) = \funct(f).
\]
\begin{proposition}\label{prop:rotinv}
For $f\in L_2(\sph{2})$ and $\ell\ge 0$ the sum over $m$ of the squares of
the absolute values of the Fourier coefficients $\Fcoe{f}$,
\[
	\ml(f):= \sum_{m=-\ell}^\ell |\Fcoe{f}|^2,
\]
is rotationally invariant.
\end{proposition}
\begin{proof}
By Fubini's theorem we can write, using \eqref{eq:Fseries},
\begin{align*}
\ml(f)= \sum_{m=-\ell}^\ell |\Fcoe{f}|^2
&= \int_{\sph{2}}\int_{\sph{2}}f(\PT{x})\conj{f(\PT{x}')}
\sum_{m=-\ell}^\ell  \conj{\shY(\PT{x})}\shY(\PT{x}')
\IntDiff[]{x}\mathrm{d}\sigma(\PT{x}')\\
&=\int_{\sph{2}}\int_{\sph{2}}f(\PT{x})\conj{f(\PT{x}')}
\frac{(2\ell+1)}{4\pi}\Legen(\PT{x}\cdot\PT{x}')
\IntDiff[]{x}\mathrm{d}\sigma(\PT{x}'),
\end{align*}
where $P_\ell$ is the Legendre polynomial scaled so that $P_\ell(1)=1$,
and in the last step we used the addition theorem for spherical harmonics
\cite{Muller1966}.
Similarly, we have
\begin{align*}
 \ml(\rtop f)
&=\int_{\sph{2}}\int_{\sph{2}}\rtop f(\PT{x})\conj{\rtop f(\PT{x}')}
\frac{(2\ell+1)}{4\pi}\Legen(\PT{x}\cdot\PT{x}')
\IntDiff[]{x}\mathrm{d}\sigma(\PT{x}')\\
&=\int_{\sph{2}}\int_{\sph{2}}f(\rho^{-1}\PT{x})\conj{f(\rho^{-1}\PT{x}')}
\frac{(2\ell+1)}{4\pi}\Legen(\PT{x}\cdot\PT{x}')
\IntDiff[]{x}\mathrm{d}\sigma(\PT{x}').
\end{align*}
Now change variables to $\PT{z}:=\rho^{-1}\PT{x}$ and
$\PT{z'}:=\rho^{-1}\PT{x'}$, and use the rotational invariance of the
inner product,
\[
\PT{x}\cdot\PT{x}' = (\rho^{-1}\PT{x})\cdot(\rho^{-1}\PT{x}') =
\PT{z}\cdot\PT{z}',
\]
together with the rotational invariance of the surface measure
to obtain
\begin{align*}
 \ml(\rtop f)
=\int_{\sph{2}}\int_{\sph{2}}f(\PT{z})\conj{f(\PT{z}')}
\frac{(2\ell+1)}{4\pi}\Legen(\PT{z}\cdot\PT{z}')
\IntDiff[]{z}\mathrm{d}\sigma(\PT{z}') = \ml(f),
\end{align*}
thus completing the proof.
\end{proof}

\subsection{Random fields on spheres}
Let $(\Omega,\cF,\probm)$ be a probability space and let $\cB(\sph{2})$
denote the Borel algebra on $\sph{2}$. A real-valued random field on the
sphere $\sph{2}$ is a function $T:\Omega \times \sph{2} \rightarrow \R$
which is measurable on $\cF \otimes \cB(\sph{2})$.  Let $\Lppsph{2}{2}$ be
the $L_{2}$ space on the product space $\prodpsph[2]$ with product measure
$\prodpsphm[]$. In the paper, we assume that $\RF\in \Lppsph{2}{2}$. By Fubini's
theorem, $\RF\in \Lp{2}{2}$ $\Pas$, in which case $T$ admits an expansion
in terms of spherical harmonics, $\Pas$,
\begin{equation}\label{eq:KL}
   T  \sim \sum_{\ell=0}^\infty \sum_{m=-\ell}^{\ell}
     \coesh \shY, \qquad  \coesh := \coesh(\omega) = \int_{\sph{2}} \RF(\PT{x}) \conj{\shY(\PT{x})}
       \IntDiff[]{x}.
\end{equation}
We will for brevity write $\RF(\omega,\PT{x})$ as $\RF(\omega)$ or $\RF(\PT{x})$ if no confusion arises.

The rotational invariance of the sum of $|\coesh|^2$ over $m$ is a corollary to Proposition \ref{prop:rotinv},
which we state as follows.
\begin{corollary}\label{cor:rot inv T}
The coefficients $\coesh$ of the random field $\RF$ in \eqref{eq:KL}
have the property that for each $\ell\ge 0$
\[
\sum_{m=-\ell}^\ell |\coesh(\omega)|^2\; \mbox{is rotationally invariant}, \quad \omega\in \probSp.
\]
\end{corollary}

The coefficients $\coesh$ are assumed to be uncorrelated mean-zero complex-valued
random variables, that is
\[
\bbE[\coesh]=0,\quad \bbE[\coesh \conj{\coesh[\ell',m']}]=
C_{\ell,m} \delta_{\ell,\ell'}\delta_{m m'},
\]
where the $C_{\ell,m}$ are non-negative numbers. The sequence
$(C_{\ell,m})$ is called the \emph{angular power spectrum} of the random field $\RF$.

It follows that $\RF(\PT{x})$ has mean zero for each $\PT{x} \in \sph{2}$ and covariance
\[
\bbE[T(\PT{x})T(\PT{y})]=
\sum_{\ell=0}^\infty \sum_{m=-\ell}^{\ell}C_{\ell,m}\conj{\shY(\PT{x})}\shY(\PT{y}),
\qquad \PT{x}, \PT{y} \in \sph{2},
\]
assuming for the moment that the sum is convergent.

In this paper we are particularly concerned with questions of isotropy.
Following \cite{MaPe2011}, the random field $T$ is \emph{strongly
isotropic} if for any $k \in \bbN$ and for any set of $k$ points
$\PT{x}_1,\ldots,\PT{x}_k \in \sph{2}$ and for any rotation $\rho \in
\RotGr[3]$, $T(\PT{x}_1), \ldots, T(\PT{x}_k)$ and $T(\rho
\PT{x}_1),\ldots,T(\rho \PT{x}_k)$ have the same law, that is, have the same joint distribution in $\probSp^{k}$.

A more easily satisfied property is weak isotropy: for an integer $n \ge
1$, $T$ is said to be $n$-\emph{weakly isotropic} if for all $\PT{x} \in
\sph{2}$, the $n$th-moment of $T(\PT{x})$ is finite, i.e. $\bbE[
|T(\PT{x})|^n ] <\infty$, and if for $k=1,\ldots,n$, for all sets of $k$
points $\PT{x}_1,\ldots,\PT{x}_k \in \sph{2}$ and for any rotation $\rho
\in \RotGr[3]$,
\[
  \bbE[T(\PT{x}_1) \cdots T(\PT{x}_k) ] = \bbE [ T(\rho \PT{x}_1) \cdots T(\rho \PT{x}_k) ].
\]

If the field $T$ is at least 2-weakly isotropic and also satisfies
$\bbE[T(\PT{x})]=0$ for all $\PT{x}\in\sph{2}$ then by definition
the covariance $\bbE[T(\PT{x})T(\PT{y})]$ is rotationally
invariant, and hence admits an $L_2$-convergent expansion in terms of Legendre
polynomials,
\begin{equation*}
\bbE[T(\PT{x})T(\PT{y})] =
\sum_{\ell=0}^\infty \frac{2\ell+1}{4\pi} \APS{\ell} \Legen(\PT{x} \cdot \PT{y})
=\sum_{\ell=0}^\infty \sum_{m=-\ell}^{\ell}\APS{\ell}\conj{\shY(\PT{x})}\shY(\PT{y}),
\end{equation*}
where in the last step we again used the addition theorem for spherical
harmonics. Thus in this case we have $C_{\ell,m} = \APS{\ell}$, and the
angular power spectrum is independent of $m$,
and can be written as
\[
C_\ell = \bbE[|a_{\ell,m}|^2 ] =
 \frac{1}{2\ell+1} \bbE\left[ \sum_{m=-\ell}^\ell |a_{\ell,m}|^2\right].
\]
We note that the scaled angular power spectrum as used in astrophysics for the CMB data,
see for example \cite{Gorski_etal2005}, is
\[
  D_\ell := \frac{\ell(\ell+1)}{2\pi} \;  \APS{\ell}. 
\]

 A random field $T$ is \emph{Gaussian} if for each $k\in \bbN$ and each choice of
$\PT{x}_1,\ldots,\PT{x}_k \in \sph{2}$ the vector $(T(\PT{x}_1),\ldots,T(\PT{x}_k))$ is
a multivariate random variable with a Gaussian distribution. A Gaussian random field is completely specified by giving its mean and covariance function.

The following proposition relates Gaussian and isotropy properties of a
random field.
\begin{proposition}\label{pro:GaussF}
\cite[Proposition 5.10]{MaPe2011} Let $T$ be a Gaussian random field on
$\sph{2}$. Then $T$ is strongly isotropic if and only if $T$ is 2-weakly
isotropic.
\end{proposition}

By \cite[Theorem 5.13, p. 123]{MaPe2011}, a $2$-weakly isotropic random
field is in $\Lp{2}{2}$ $\Pas$.

In the present paper we are principally concerned with input random fields that are both Gaussian and strongly isotropic.
Our main aim is to show that the resulting regularized field is also strongly isotropic.
(Of course the Gaussianity of the field is inevitably lost, given that some of the coefficients may be replaced by zero.)

\subsection{Norms and regularization models}\label{sec:reg}
In this section the randomness of the field plays no real role.
Thus the observed field $\RFo$ may be thought of either as a deterministic
field or as one realization of a random field.
Assume that the observed field $\RFo$ is given by
\begin{equation}\label{eq:RF.o}
   \RFo = \sum_{\ell=0}^\infty \sum_{m=-\ell}^{\ell} \coesho \shY.
\end{equation}
Consider an approximating field $\RF$ with the spherical harmonic expansion
\begin{equation*}
   \RF = \sum_{\ell=0}^\infty \sum_{m=-\ell}^{\ell} \coesh \shY.
\end{equation*}
Let
$\Nz = \{0, 1, 2, \ldots\}$ and let
\begin{equation}\label{eq:Aell}
  A_\ell := \left(\sum_{m=-\ell}^\ell | \coesh |^2 \right)^\frac{1}{2}, 
  \quad \ell \in\Nz.
\end{equation}
Then
\[
\|T\|_{L_2(\mathbb{S}^2)}^2=
\sum_{\ell=0}^\infty\sum_{m=-\ell}^\ell|a_{\ell,m}|^2 = \sum_{\ell=0}^\infty A_\ell^2.
\]
Clearly $A_\ell = 0$ if and only if $\coesh = 0$ for $m =
-\ell,\ldots,\ell$.


For simplicity, we let $\vcoesh:=\vcoesh(\RF)$ (an infinite dimensional vector) denote the
sequence of spherical harmonic coefficients $\coesh, m = -\ell,\ldots,\ell$,
$\ell\in\Nz$, of the field $\RF$:
\begin{equation}\label{eq:avec}
  \vcoesh := (\coesh[0,0], \coesh[1,-1], \coesh[1,0], \coesh[1,1], \ldots,
 \coesh[\ell,-\ell], \ldots, \coesh[\ell,\ell],\ldots)^T.
\end{equation}
For a positive sequence $\{\beta_{\ell}\}_{\ell\in\Nz}$, we define the norm
\begin{equation} \label{eq:norm}
  \norm{\vcoesh}{1,2,\beta} := \norm{\vcoesh}{1,2} :=
  \sum_{\ell=0}^\infty \beta_\ell \left( \sum_{m=-\ell}^\ell |\coesh|^2 \right)^\frac{1}{2} =
  \sum_{\ell=0}^\infty \beta_\ell A_\ell.
\end{equation}
We call $\beta_{\ell}$ the \emph{degree-scaling} sequence, because it describes the
relative importance of different degrees $\ell$.
(In Section~\ref{sec:beta}, we will discuss the choice
of the parameters $\beta_\ell$.) This choice of norm, a
scaled hybrid between the standard $\ell_1$ and $\ell_{2}$ norms, is the key to
preserving isotropy while still giving sparse solutions.

We will measure the agreement between the observed data $\coesho$ and the
approximation $a_{\ell,m}$ by the $\ell_{2}$ norm, or its square, the discrepancy,
\begin{equation*}
  \norm{\vcoesh - \vcoesho}{2}^2 =
  \sum_{\ell=0}^\infty \sum_{m=-\ell}^\ell | \coesh - \coesho |^2 = \|\RF-\RFo\|^2_{\Lp{2}{2}}.
\end{equation*}

Given the observed data $\coesho$ for $\ell \in\Nz$, $m
= -\ell,\ldots,\ell$ arranged in the vector $\vcoesho$ as in
(\ref{eq:avec}), and a regularization parameter $\lambda \geq 0$, our
regularized problem is\
\begin{equation}\label{eq:model1}
  \mathop{\mathrm{Minimize}}_{\vcoesh} \
  \tfrac{1}{2}\norm{\vcoesh - \vcoesho}{2}^2 + \lambda \norm{\vcoesh}{1,2}.
\end{equation}
As both norms are convex functions and $\lambda \geq 0$, the objective is strictly convex
and there is a unique global minimizer.
Moreover, first order optimality conditions are both necessary and sufficient for
a global minimizer (see \cite{Ber2009} for example).

A closely related model is
\begin{equation}\label{eq:model2}
\begin{array}{cl}
  \displaystyle\mathop{\mathrm{Minimize}}_{\vcoesh} &
  \norm{\vcoesh}{1,2} \\[1ex]
  \mbox{Subject to} & \norm{\vcoesh - \vcoesho}{2}^2 \leq \sigma^2.
\end{array}
\end{equation}
Again, as the feasible region is bounded a global solution exists, and as
the norms are convex functions any local minimizer is a global minimizer
and the necessary conditions for a local minimizer are also sufficient.
When the constraint in (\ref{eq:model2}) is active, the Lagrange
multiplier determines the value of $\lambda$ in (\ref{eq:model1}). If the
objective was $\| \vcoesh \|_1$, instead of $\| \vcoesh \|_{1,2}$, this
would be a very simple example of the constrained $\ell_{1}$-norm minimization
problem, widely used, see \cite{Don2006,CanRT2006a,vdBerFri2011} for
example, to find sparse solutions to under-determined systems of linear
equations. Such problems, with a separable structure, can be readily
solved, see \cite{WriNF2009,spgl1:2007} for example.

An alternative formulation would be a LASSO~\cite{Tib1996,OsbPT2000a} based approach:
\begin{equation}\label{eq:model3}
\begin{array}{cl}
  \displaystyle\mathop{\mathrm{Minimize}}_{\vcoesh} &
  \tfrac{1}{2} \norm{\vcoesh - \vcoesho}{2}^2 \\[1ex]
  \mbox{Subject to} & \norm{\vcoesh}{1,2} \leq \kappa.
  \end{array}
\end{equation}
Such problems, using the standard $\ell_{1}$ norm $\| \vcoesh\|_1$ instead of $\| \vcoesh\|_{1,2}$,
and related problems have been widely explored in statistics and compressed sensing,
see \cite{EfHaJoTi2004,CanRT2006a,Don2006} for example.

The \emph{regularized random field} $\RFr$ is given in terms of the
spherical harmonic expansion
\begin{equation}\label{eq:RF.r}
  \RFr := \sum_{\ell=0}^\infty \sum_{m=-\ell}^{\ell} \coeshr \shY,
\end{equation}
where the regularized coefficients $\coeshr$ minimize one of the model problems
(\ref{eq:model1}), (\ref{eq:model2}) or (\ref{eq:model3}).

We will concentrate on the model (\ref{eq:model1}). The relation to the
other models is detailed in the appendix. It is up to the user to choose
which regularization model is easiest to interpret: in a particular
application specifying a bound $\norm{\vcoesh - \vcoesho}{2}^2
\leq \sigma^2$ on the discrepancy or a bound $\norm{\vcoesh}{1,2} \leq
\kappa$ on the norm of the regularized solution may be easier to interpret than directly specifying the regularization parameter $\lambda$. The
appendix shows how to determine the corresponding value of the
regularization parameter $\lambda$ given either $\sigma$ or $\kappa$ for
these alternative models.

\section{Analytic solution to the sparse regularization model}\label{sec:regsol}
Consider the optimization problem (\ref{eq:model1}). The coefficients
$\coesh$ and $\coesho$ are complex, while all the other quantities, such
as $A_\ell, \Alo$, $\beta_\ell$ and $\alr$, are real.  Temporarily we
write the real and imaginary parts of $a_{\ell,m}$ explicitly,
\[
a_{\ell,m}=x_{\ell,m}+ \imu\hspace{0.3mm} y_{\ell,m},\quad \mbox{and define} \quad
\nabla_{a_{\ell,m}}
=\frac{\partial}{\partial x_{\ell,m}}+ \imu\hspace{0.3mm} \frac{\partial}{\partial y_{\ell,m}}.
\]
Then for all degrees $\ell$ for which $A_\ell$ is positive the
definition~(\ref{eq:Aell}) gives
\begin{equation}\label{eq:normderiv}
  \nabla_{\coesh} A_\ell =  A_\ell^{-1} \; \coesh, \qquad
  m = -\ell,\ldots,\ell,
\end{equation}
and hence from \eqref{eq:norm}
\begin{equation*}
  \nabla_{\coesh}\norm{\vcoesh}{1,2} = \beta_\ell A_\ell^{-1} \; \coesh, \qquad
  m = -\ell,\ldots,\ell,
\end{equation*}
It follows that the necessary and sufficient conditions for a local/global
minimum in \eqref{eq:model1} are
\begin{eqnarray}
 (\coesh - \coesho) + \lambda \beta_\ell A_\ell^{-1} \coesh  = 0, & & \quad
     m = -\ell,\ldots,\ell \qquad \mbox{when } A_\ell > 0, \label{eq:opt1a}\\
 \coesh = 0, & & \quad m = -\ell,\ldots,\ell \qquad \mbox{when } A_\ell = 0. \notag 
\end{eqnarray}
For each $\lambda\ge 0$ we define the degree sets
\begin{equation}\label{eq:Gamma}
\Gamma(\lambda) := \left\{\ell\in\Nz: \frac{\Alo}{\beta_\ell} > \lambda\right\},
\qquad
\Gamma^c(\lambda) := \left\{\ell\in\Nz: \frac{\Alo}{\beta_\ell} \leq \lambda\right\}.
\end{equation}
Note that both $\Gamma(\lambda)$ and $\Gamma^c(\lambda)$ are random sets.  For a particular realization of the field and for
$\lambda = 0$ the set $\Gamma^c(0)$ consists only of those
$\ell$ values for which $\Alo = 0$. %
For $\lambda \geq \sup_{\ell\in\Nz}
\{\Alo/\beta_\ell\}$
the set $\Gamma(\lambda)$ is empty. %
For degrees $\ell\in\Gamma^c(\lambda)$, all the regularized coefficients
are zero. For degrees $\ell \in \Gamma(\lambda)$ where $A_\ell > 0$,
equation (\ref{eq:opt1a}) shows that the regularized coefficients are given by
\begin{equation}\label{eq:Aellr}
 \coeshr = \alr \coesho \quad \mbox{and hence}\quad
\Alr := \left(\sum_{m=-\ell}^\ell|\coeshr|^2\right)^{1/2}  = \alr \Alo,
\end{equation}
where $\Alo$ is given by \eqref{eq:Alo},
and (\ref{eq:opt1a}) gives
\begin{equation}\label{eq:alpha}
  \alr = \frac{\Alr}{\Alr + \lambda \beta_\ell} = \frac{\Alr}{\Alo} =
  \frac{\Alo - \lambda \beta_\ell}{\Alo} \in (0, 1].
\end{equation}

Summing up, $\Alr$ is given by
\begin{equation*} \label{eq:Alr}
\Alr := \left\{ \begin{array}{cl} \displaystyle \Alo -
\lambda \beta_\ell & \quad\mbox{for }
  \ell \in \Gamma(\lambda),\\
  \displaystyle 0 & \quad\mbox{for } \ell \in \Gamma^c(\lambda),
  \end{array}\right.
\end{equation*}
and the regularized coefficients are
\begin{equation*}
  \coeshr = \left\{\begin{array}{cl}
\alr \coesho & \mbox{for } m = -\ell,\ldots,\ell, \quad \ell \in
\Gamma(\lambda),\\
0 & \mbox{for } m = -\ell,\ldots,\ell, \quad \ell\in\Gamma^c(\lambda).
  \end{array}\right.
\end{equation*}
In the vector notation introduced in \eqref{eq:avec},
\begin{eqnarray}
  & & \norm{\vcoesh^{\mathrm{r}}}{1,2} = 
      \sum_{\ell\in\Gamma(\lambda)} \beta_\ell (\Alo - \lambda\beta_\ell) =\sum_{\ell\in\Gamma(\lambda)}\beta_\ell A_\ell^r, \label{eq:lamnrm}\\[1ex]
  & & \norm{\vcoesh^{\mathrm{r}} - \vcoesho}{2}^2 =
      \lambda^2\sum_{\ell\in\Gamma(\lambda)}\beta_\ell^2
      +\sum_{\ell\in\Gamma^c(\lambda)} (\Alo)^2. \label{eq:lamssq}
\end{eqnarray}
The value $\lambda = 0$ gives the solution $\vcoesh^{\mathrm{r}} =
\vcoesho$ (noting that in \eqref{eq:lamssq} the first term vanishes for
$\lambda=0$ while in the second sum each term is zero).  For $\lambda
> \sup_{\ell\in \Nz} \{\Alo/\beta_\ell\}$ the solution is
$\vcoesh^{\mathrm{r}} = 0$.
We have established the solution to problem \eqref{eq:model1}, as summarized by the following proposition.
\begin{proposition}\label{prop:reg.sol}
Let $\coesho$, $m=-\ell,\dots,\ell$, $\ell\in\Nz$, be the Fourier coefficients for a random field $\RFo$ on $\sph{2}$. For a positive sequence $\{\beta_{\ell}\}_{\ell\in\Nz}$ and a positive regularization parameter $\lambda$, the solution of the regularization problem \eqref{eq:model1} is, in the $\Lppsph{2}{2}$ sense,
\begin{equation*}
  \RFr = \sum_{\ell=0}^\infty \sum_{m=-\ell}^{\ell} \coeshr \shY
\end{equation*}
with regularized coefficients, for $m=-\ell,\dots,\ell$ and $\ell\in\Nz$,
\begin{equation*}
   \coeshr := \left\{\begin{array}{ll}
	\displaystyle \left(1 - \frac{\lambda\beta_{\ell}}{\Alo}\right) \coesho, & \Alo>\lambda\beta_{\ell},\\
	0, & \Alo\le\lambda\beta_{\ell},
 \end{array}\right.
\quad
\Alr = \left\{\begin{array}{ll}
        \displaystyle \left(1 - \frac{\lambda\beta_{\ell}}{\Alo}\right) \Alo, & \Alo>\lambda\beta_{\ell},\\
        0, & \Alo\le\lambda\beta_{\ell},
 \end{array}\right.
\end{equation*}
where $\Alo$ is given by \eqref{eq:Alo}, and $A_\ell^r$ by \eqref{eq:Aellr}. 
\end{proposition}


\section{Regularization preserves strong isotropy}\label{sec:iso}
Marinucci and Peccati \cite[Lemma 6.3]{MaPe2011} proved that the Fourier
coefficients of a strongly isotropic random field have the same law under any rotation of the coordinate axes, in a sense to be made precise in the first part of the following theorem. In the following theorem we prove that the converse is also
true.

In the theorem, $D^\ell(\rho)$, for a given $\ell\ge 0$ and a given
rotation $\rho\in \RotGr[3]$, is the $(2\ell+1)\times (2\ell+1)$  Wigner
matrix, which has the property
\[
\sum_{m'=-\ell}^\ell D^\ell(\rho)_{m',m} \shY[\ell,m'](\PT{x})
=\shY( \rho^{-1} \PT{x}),\quad m=-\ell, \ldots, \ell.
\]
The Wigner matrices form (irreducible) $(2\ell+1)$-dimensional
representations of the
 rotation group $\RotGr[3]$, in the sense that (as can easily be verified)
\[
D^\ell(\rho_1)D^\ell(\rho_2)=D^\ell(\rho_1\rho_2), \quad \rho_1,\rho_2 \in \RotGr[3].
\]
\begin{theorem}
\label{lem:aell isotropy} Let $\RF$ be a real, square-integrable random
field on $\sph{2}$, with the spherical harmonic coefficients $\coesh$. Let
$\vcoesh_{\ell \cdot}$  denote the corresponding
$(2\ell+1)$-dimensional vector,
$\vcoesh_{\ell \cdot}:=(\coesh[\ell ,-\ell],\ldots,\coesh[\ell
,\ell])^T$.\\
(i) \cite[Lemma 6.3]{MaPe2011} If $T$ is strongly isotropic then for every
rotation $\rho\in\RotGr[3]$, every $k\ge 1$ and every
$\ell_1,\ldots,\ell_k\ge 0$, we have
\begin{equation}\label{cond aell}
   (D^{\ell_1}(\rho)  \vcoesh_{\ell_1 \cdot}, \ldots, D^{\ell_k}(\rho)  \vcoesh_{\ell_k \cdot})
   \overset{d}{=} (\vcoesh_{\ell_1 \cdot},\ldots,\vcoesh_{\ell_k \cdot}),
\end{equation}
where $\overset{d}{=}$ denotes identity in distribution.\\
(ii) If the condition \eqref{cond aell} holds for all $\rho\in \RotGr[3]$, all $k\ge1$ and any $\ell_1,\dots,\ell_{k}\ge0$, then the field $\RF$ is
strongly isotropic.
\end{theorem}

\begin{proof}[Proof of (ii)]
Let $\rho$ be a rotation in $\RotGr[3]$ and let $\PT{x}_1,\ldots,\PT{x}_k$ be
$k$ arbitrary points on $\sph{2}$. Then
\begin{align*}
 & \left( \RF(\rho^{-1}\PT{x}_1),\ldots,\RF(\rho^{-1}\PT{x}_k) \right)\\
 & = \left(
\sum_{\ell_1=0}^\infty \sum_{m_1=-\ell_1}^{\ell_1} a_{\ell_1,m_1} Y_{\ell_1,m_1}( \rho^{-1} \PT{x}_1),\ldots,
\sum_{\ell_k=0}^\infty \sum_{m_k=-\ell_k}^{\ell_k} a_{\ell_k,m_k} Y_{\ell_k,m_k}( \rho^{-1} \PT{x}_k)
 \right) \\
& =  \left(
\sum_{\ell_1=0}^\infty \sum_{m_1=-\ell_1}^{\ell_1} a_{\ell_1,m_1} \sum_{m_1'=-\ell_1}^{\ell_1} D^{\ell_1}(\rho)_{m_1',m_1} \shY[\ell_1,m_1'](\PT{x}_1),\ldots, \right.\\
&\left.\qquad\qquad\sum_{\ell_k=0}^\infty \sum_{m_k=-\ell_k}^{\ell_k} a_{\ell_k,m_k} \sum_{m_k'=-\ell_k}^{\ell_k} D^{\ell_k}(\rho)_{m_k',m_k} \shY[\ell_k,m_k'](\PT{x}_k)
\right) \\
& =  \left(
\sum_{\ell_1=0}^\infty \sum_{m_1'=-\ell_1}^{\ell_1} \wtd{\coesh[\ell_1,m_1']}  \shY[\ell_1,m_1'](\PT{x}_1),\ldots,
\sum_{\ell_k=0}^\infty \sum_{m_k'=-\ell_k}^{\ell_k} \wtd{\coesh[\ell_k,m_k']}  \shY[\ell_k,m_k'](\PT{x}_k)
\right)
\end{align*}
where we write
\[
\wtd{ \coesh[\ell,m']  } := \sum_{m=-\ell}^\ell D^{\ell}(\rho)_{m',m}\: \coesh.
\]
Since condition \eqref{cond aell} holds, for all $\ell_1,\ldots,\ell_k\ge
0$ we have
\[
(\wtd{\coesh[\ell_1,-\ell_1]},\ldots, \wtd{\coesh[\ell_1,\ell_1]}, \ldots,
\wtd{\coesh[\ell_k,-\ell_k]},\ldots,\wtd{\coesh[\ell_k,\ell_k]}) \overset{d}{=}
(\coesh[\ell_1,-\ell_1],\ldots, \coesh[\ell_1,\ell_1]
\ldots, \coesh[\ell_k,-\ell_k],\ldots, \coesh[\ell_k,\ell_k]).
\]
Now we use a simple
 instance of the
 principle that if a finite set $B$ of random variables has the same joint
 distribution as another set $B'$, then, for any measurable real-valued function $f$, $f(B)$
 will have the same joint distribution as $f(B')$.
Thus,
\begin{align*}
  \left( \RF(\rho^{-1}\PT{x}_1),\ldots,\RF(\rho^{-1}\PT{x}_k) \right)
& \overset{d}{=}
\left(
\sum_{\ell_1=0}^\infty \sum_{m_1=-\ell_1}^{\ell_1} \coesh[\ell_1,m_1] \shY[\ell_1,m_1](\PT{x}_1),\ldots,
\sum_{\ell_k=0}^\infty \sum_{m_k=-\ell_k}^{\ell_k} \coesh[\ell_k,m_k] \shY[\ell_k,m_k](\PT{x}_k)
 \right) \\
& = \left( \RF(\PT{x}_1),\ldots,\RF(\PT{x}_k) \right).
\end{align*}
In other words, the random field $T$ is strongly isotropic.
\end{proof}

The following theorem shows that the regularized random field $\RFr$ in
\eqref{eq:RF.r} is strongly isotropic if the observed random field $\RFo$
is strongly isotropic.

\begin{theorem}\label{thm:isotropy.reg.RF} Let $\RFo$ be a real observed random field on the sphere $\sph{2}$ as in \eqref{eq:RF.o} and let $\RFr$ given by 
Proposition~\ref{prop:reg.sol} be the correspondingly regularized random field. If $\RFo$ is strongly isotropic then the regularized random field $\RFr$ is also strongly isotropic.
\end{theorem}

\begin{proof}
For an arbitrary realization of the regularized field we have
\begin{align}\label{eq:RF.r.a}
  \RFr(\PT{x}) & = \sum_{\ell=0}^\infty \sum_{m=-\ell}^{\ell} \coeshr \shY(\PT{x})\notag\\[1mm]
  & = \sum_{\ell=0}^\infty \sum_{m=-\ell}^{\ell} \alr(\RFo)\coesho \shY(\PT{x})\notag,\quad \PT{x}\in\sph{2},\quad  
\end{align}
where the $\alpha_\ell(\RFo)$, for $\ell=0,1,2,\ldots$ given by \eqref{eq:alpha}, are rotationally invariant
as a consequence of Corollary~\ref{cor:rot inv T}.

Since $\RFo$ is strongly isotropic, from Theorem~\ref{lem:aell isotropy} part (i),
for any rotation $\rho \in \RotGr[3]$, every $k \ge 1$ and every $\ell_1,\ldots,\ell_k \ge 0$, we have
\[
   (D^{\ell_1}(\rho)  \vcoesho_{\ell_1 \cdot}, \ldots, D^{\ell_k}(\rho)  \vcoesho_{\ell_k \cdot})
   \overset{d}{=} (\vcoesho_{\ell_1 \cdot},\ldots,\vcoesho_{\ell_k \cdot}).
\]
It follows from the rotational invariance of the $\alpha_{\ell}$ that
\begin{equation}\label{eq:D.a}
 (\alpha_{\ell_1} D^{\ell_1}(\rho) \vcoesho_{\ell_1 \cdot},\ldots,
 \alpha_{\ell_k} D^{\ell_k}(\rho) \vcoesho_{\ell_k \cdot})
\overset{d}{=}(\alpha_{\ell_1} \vcoesho_{\ell_1 \cdot}, \ldots,
\alpha_{\ell_k} \vcoesho_{\ell_k \cdot})
\end{equation}
The equality in \eqref{eq:D.a} is equivalent to
\[
  (D^{\ell_1}(\rho) \vcoeshr_{\ell_1 \cdot},\ldots,
   D^{\ell_k}(\rho) \vcoeshr_{\ell_k \cdot}) \overset{d}{=}
(\vcoeshr_{\ell_1 \cdot},\ldots, \vcoeshr_{\ell_k \cdot}),
\]
for any rotation $\rho$, every $k \ge 1$ and every $\ell_1,\ldots,\ell_k \ge 0$.
So, by Theorem~\ref{lem:aell isotropy} part (ii) the field $\RFr$ is strongly isotropic.
\end{proof}

The above theorem and Proposition~\ref{pro:GaussF} imply the following corollary.
\begin{corollary}
The regularized random field $\RFr$ is strongly isotropic if the
observed random field $\RFo$ is Gaussian and 2-weakly isotropic.
\end{corollary}


\section{Approximation error of the regularized solution}\label{sec:err}
This section
estimates the approximation error of the sparsely regularized random field
from the observed random field, and gives one choice for the regularization
parameter $\lambda$.

Let $\{\coesho|\ell\in\Nz,\:m=-\ell,\dots,\ell\}$ and $\{\coeshr|\ell\in\Nz,\:m=-\ell,\dots,\ell\}$ be the Fourier coefficients for an observed random field $\RFo$ and the regularized field $\RFr$ on $\sph{2}$ respectively.

\begin{lemma} Let $\RFo$ be a random field in $\Lppsph{2}{2}$.
For any $\lambda>0$ and any positive sequence $\{\beta_{\ell}\}_{\ell=0}^{\infty}$,
let $\RFr$ be the regularized solution to the
{regularization problem (\ref{eq:model1}) with regularization parameter $\lambda$.}
Then $\RFr$ is in $\Lppsph{2}{2}$.	
\end{lemma}
\begin{proof}
	By \eqref{eq:alpha}, $0<\alpha_{\ell}\le1$ for $\ell\in\Gamma(\lambda)$. We now define $\alpha_{\ell}=0$ for $\ell\in \Gamma^{c}(\lambda)$, so that $\alpha_{\ell}\in[0,1]$ for all $\ell\in\Nz$. Since $\RFo$ is in $\Lppsph{2}{2}$, by Parseval's identity and Fubini's theorem,
\begin{align*}
	\norm{\RFr}{\Lppsph{2}{2}}^{2}=\expect{\norm{\RFr}{\Lp{2}{2}}^{2}}
	&=\expect{\sum_{\ell=0}^{\infty}\sum_{m=-\ell}^{\ell} |\coeshr|^2}
	=\expect{\sum_{\ell=0}^{\infty}\sum_{m=-\ell}^{\ell} |\alpha_{\ell}\coesho|^2}\\
	&\le\expect{\sum_{\ell=0}^{\infty}\sum_{m=-\ell}^{\ell} |\coesho|^2}
	=\expect{\norm{\RFo}{\Lp{2}{2}}^{2}}
	=\norm{\RFo}{\Lppsph{2}{2}}^{2}<\infty.
\end{align*}
Thus, $\RFr$ is in $\Lppsph{2}{2}$.
\end{proof}

The following theorem shows that the $\Lppsph{2}{2}$ error of the regularized solution can be arbitrarily small with an appropriate regularization parameter $\lambda$.
\begin{theorem}\label{thm:L2err.RFr}
Let $\RFo$ be a random field in $\Lppsph{2}{2}$.
For any $\eps>0$ and any positive sequence $\{\beta_{\ell}\}_{\ell=0}^{\infty}$,
let $\RFr$ be the regularized field of the solution to the regularization problem
{(\ref{eq:model1}) with regularization parameter satisfying}
$0\le\lambda<\frac{\eps}{2\sqrt{\sum_{\ell=0}^{\ell^{*}}\beta_{\ell}^{2}}}$, where $\ell^{*}$ is the smallest integer such that $\sum_{\ell>\ell^{*}} \expect{(\Alo)^2} \le \eps^{2}/4$, where $\Alo$ is given by \eqref{eq:Alo}. Then,
\begin{equation}\label{eq:Lppsph2.err}
	\norm{\RFo - \RFr}{\Lppsph{2}{2}}  < \eps.
\end{equation}
\end{theorem}
\begin{remark}
	The integer $\ell^{*}$ in the theorem exists as the series $\sum_{\ell=0}^{\infty}\expect{(\Alo)^2}=\norm{\RFo}{\Lppsph{2}{2}}^{2}$ is convergent.
\end{remark}

\begin{proof}
Using Fubini's theorem and the degree sets defined in \eqref{eq:Gamma},
we split the squared $\Lppsph{2}{2}$ error of the regularized field $\RFr$ 
as
\begin{align}\label{eq:I0}
 \norm{\RFo - \RFr}{\Lppsph{2}{2}}^{2}
 &= \expect{\norm{\RFo - \RFr}{\Lp{2}{2}}^{2}}\\
  &=
  \expect{\sum_{\ell=0}^{\infty}\sum_{m=-\ell}^{\ell} \bigl|\coesho - \coeshr\bigr|^2}\notag\\
   &=\expect{\sum_{\ell \in \Gamma^c(\lambda)} \sum_{m=-\ell}^{\ell} |\coesho|^2
    + \sum_{\ell \in \Gamma(\lambda)} 
   \sum_{m=-\ell}^{\ell} \bigl|(1-\alpha_{\ell}) \coesho\bigr|^2}\notag\\
   &=\expect{\sum_{\ell \in \Gamma^c(\lambda)} (\Alo)^2}
    + \expect{\sum_{\ell \in \Gamma(\lambda)}\bigl|(1-\alpha_{\ell}) \Alo\bigr|^2},\notag
\end{align}
where the second equality is by Parseval's identity, 
the third equality uses
equation \eqref{eq:Aellr} and the fourth equality 
uses \eqref{eq:Alo}.

Since $\RFo$ is in $\Lppsph{2}{2}$,
\begin{align*}
	\sum_{\ell=0}^{\infty} \expect{(\Alo)^2}=\norm{\RFo}{\Lppsph{2}{2}}^{2}<\infty,
\end{align*}
thus there exists the smallest integer $\ell^{*}$ such that
\begin{equation*}
	\expect{\sum_{\ell>\ell^{*}}(\Alo)^2}=\sum_{\ell>\ell^{*}}\expect{(\Alo)^2} \le \frac{\eps^{2}}{4}.
\end{equation*}
This shows that the first term of the right-hand side of \eqref{eq:I0} is bounded above by
\begin{align}
\expect{\sum_{\ell \in \Gamma^c(\lambda)}(\Alo)^2}
 &= \expect{\sum_{\ell \le \ell^{*}; \ell \in \Gamma^c(\lambda)} (\Alo)^2
      + \sum_{\ell> \ell^{*}; \ell \in \Gamma^c(\lambda)} (\Alo)^{2}} \notag\\
 &\le \expect{\sum_{\ell \le \ell^{*}; \ell \in \Gamma^c(\lambda)} (\Alo)^2}
      + \expect{\sum_{\ell> \ell^{*}} (\Alo)^{2}} \notag\\
 &\le \expect{\sum_{\ell\le\ell^{*}; \ell \in \Gamma^c(\lambda)}      (\Alo)^2}
      + \frac{\eps^{2}}{4}.\label{eq:I1}
\end{align}
For the first term of the right-hand side of \eqref{eq:I1}, we have $\Alo\le \lambda\beta_{\ell}$, and hence
\begin{align}\label{eq:I1.1}
\expect{\sum_{\ell\le\ell^{*};\ell \in \Gamma^c(\lambda)} (\Alo)^{2}}
\le \expect{\sum_{\ell\le\ell^{*};\ell \in \Gamma^c(\lambda)} (\lambda\beta_{\ell})^{2}}
&\le \lambda^2 \sum_{\ell=0}^{\ell^{*}} \beta^2_\ell
   < \frac{\eps^{2}}{4},
\end{align}
where we used the condition
\begin{equation*}
 \lambda <\frac{\eps}{2\sqrt{\sum_{\ell=0}^{\ell^{*}}\beta_{\ell}^{2}}}.
\end{equation*}

We now estimate the second term of the right-hand side of \eqref{eq:I0}. By \eqref{eq:alpha} for $\ell\in\Gamma(\lambda)$ we have $1-\alpha_{\ell}=\lambda\beta_{\ell}/\Alo \le 1$, thus
\begin{align*}
 \expect{\sum_{\ell \in \Gamma(\lambda)}\bigl| (1-\alpha_{\ell}) \Alo\bigr|^{2}}
 &\le \expect{\sum_{\ell\le\ell^{*};\ell\in\Gamma(\lambda)}(\lambda\beta_{\ell})^{2} + \sum_{\ell>\ell^{*};\ell\in\Gamma(\lambda)}(\Alo)^{2}}\notag\\
 &\le \sum_{\ell=0}^{\ell^{*}}\lambda^{2}\beta_{\ell}^{2} + \expect{\sum_{\ell>\ell^{*}}(\Alo)^2}
  <\frac{\eps^{2}}{4} + \frac{\eps^{2}}{4}
  < \frac{\eps^{2}}{2}.
\end{align*}
This with \eqref{eq:I1.1}, \eqref{eq:I1} and \eqref{eq:I0} gives \eqref{eq:Lppsph2.err}.
\end{proof}


\section{Scaling to preserve the $L_{2}$ norm}\label{sec:scaling}
The sparse regularization leads
to a reduction of the $L_2$-norm of the regularized field
from that of the observed field.
In this section, we scale the regularized field so that the $L_{2}$-norm of the resulting
field is preserved. 


By \eqref{eq:Aell} and Parseval's identity,
\begin{equation*}
\begin{array}{ll}
	\norm{\RFo}{\Lp{2}{2}}^{2}&=\displaystyle\sum_{\ell=0}^{\infty}
            \sum_{m=-\ell}^{\ell}|\coesho|^{2}=\sum_{\ell=0}^{\infty}(\Alo)^{2},\\[5mm]
	\norm{\RFr}{\Lp{2}{2}}^{2}&=\displaystyle\sum_{\ell=0}^{\infty}
          \sum_{m=-\ell}^{\ell}|\coeshr|^{2}=\sum_{\ell=0}^{\infty}(\Alr)^{2}.
\end{array}
\end{equation*}


For each realization $\RFo(\omega)$, $\omega\in\probSp$ of an observed field $\RFo$, we define a new random variable,
the \emph{scaling (factor) for the $L_{2}$ norm}, by
\begin{equation}\label{eq:scalf}
	\scalf:=\scalf(\omega):=\scalf(\RFo(\omega),\RFr(\omega)):=\frac{\norm{\RFo(\omega)}{\Lp{2}{2}}}{\norm{\RFr(\omega)}{\Lp{2}{2}}}=\sqrt{\sum_{\ell=0}^{\infty}(\Alo)^{2}/\sum_{\ell=0}^{\infty}(\Alr)^{2}}.
\end{equation}
Then, for the same realization, we scale up the regularized field $\RFr$ by multiplying
by the factor $\scalf$ to obtain
\begin{equation*}
	\widetilde{\RFr} := \scalf\: \RFr.
\end{equation*}
We say the resulting field $\widetilde{\RFr}$ is the \emph{scaled regularized field} of $\RFo$ for the parameter choices $\lambda$ and $\{\beta_{\ell}\}_{\ell=0}^{\infty}$.


\section{Numerical experiments}\label{sec:num}

In this section, we use cosmic microwave background (CMB) data on $\sph{2}$,
see for example \cite{Planck2016I},
to illustrate the regularization algorithm.

\subsection{CMB data}\label{sec:CMBdata}

The CMB data giving the sky temperature of cosmic microwave background are available on $\sph{2}$
at HEALPix points
(Hierarchical Equal Area isoLatitude Pixelation)
\footnote{\url{http://healpix.sourceforge.net}} \cite{Gorski_etal2005}. These points provide
an equal area partition of $\sph{2}$ and are equally spaced on rings of constant latitude.
This enables the use of fast Fourier transform (FFT) techniques for spherical harmonics.

In the experiments, we use the CMB map
with $N_{\rm side} = 2048$, giving $N_{\rm pix}=12\times 2048^2=50,331,648$ HEALPix points,
see \cite{Planck2016IX}, as computed by SMICA \cite{CaLeDeBePa2008}, a component separation
method for CMB data processing, see Figure~\ref{fig:CMB_Original}.
In this map the mean $a^{\rm o}_{0,0}$ and first moments $a^{\rm o}_{1,m}$, for $m=-1,0,1$ are set to zero.
A CMB map can be modelled as a realization of a strongly isotropic
random field $\cmbRF$ on $\sph{2}$.
\begin{figure}[ht]
\centering
\scalebox{0.6}{
\includegraphics[trim = 0mm 0mm 0mm 0mm, width=\textwidth]{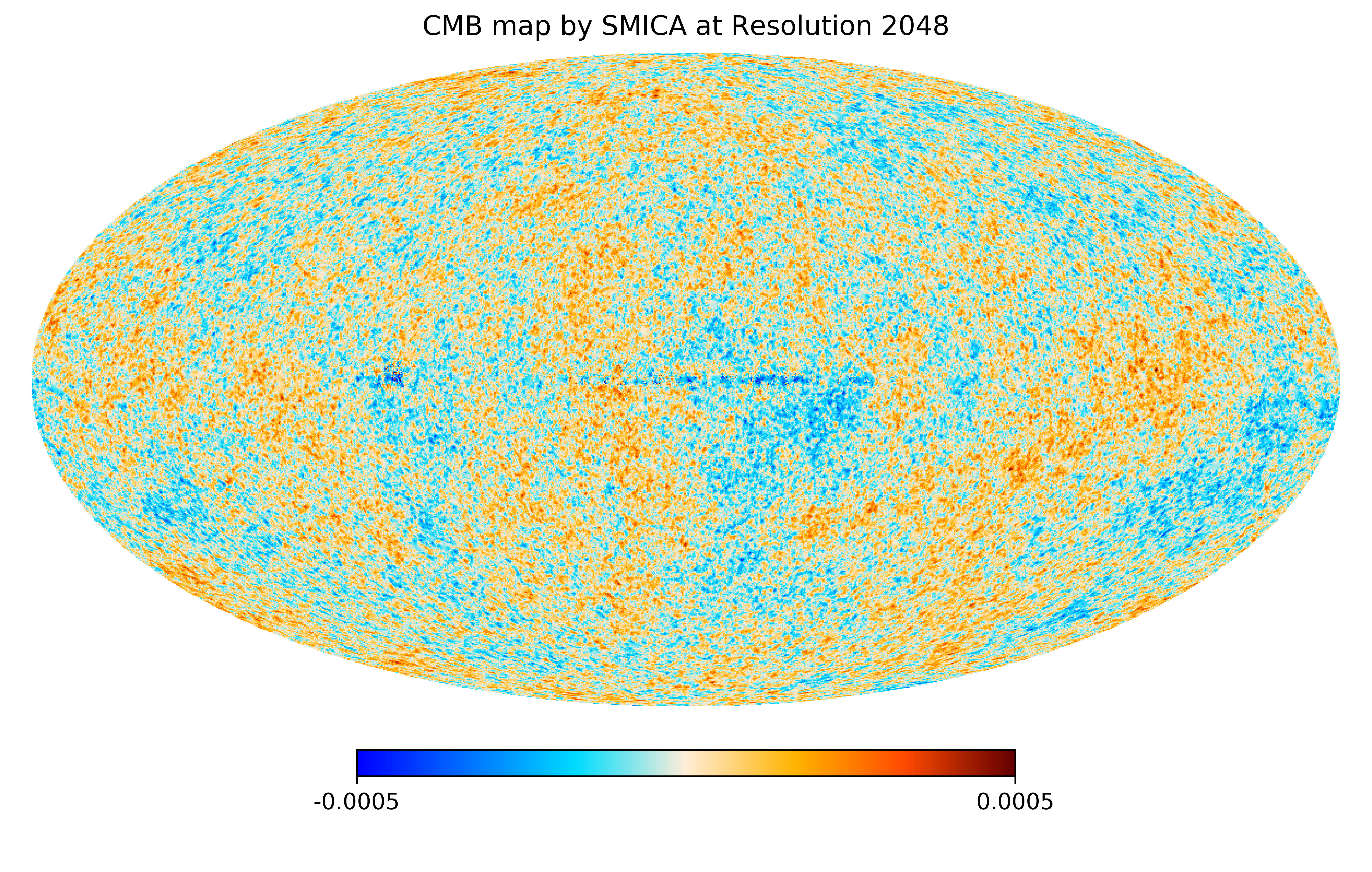}
}
\caption{The CMB data with $N_{\rm side} = 2048$ as computed by SMICA.}
\label{fig:CMB_Original}
\end{figure}

\subsection{Analysis of the CMB data}

The Python HEALPy package \cite{Gorski_etal2005} was
used to calculate the Fourier
coefficients $\coesho$ of the observed field,
using an equal weight quadrature rule at the HEALPix points. This instance of CMB
data is band-limited with maximum degree $L=4,000$, thus
\begin{equation*}
\RFo =	\tRFo =\sum_{\ell=0}^{\trdeg}\sum_{m=-\ell}^{\ell}\coesho\shY.
\end{equation*}
The observed $\Alo$ given by \eqref{eq:Alo} for $\ell=0,\dots,L$
are shown on a logarithmic scale in Figure~\ref{fig:CMB_Alo} for degree $\ell$ up to $4,000$.

Once $\lambda$ and $\beta_\ell$ are chosen we easily calculate $a_{\ell,m}^r$ and $A_\ell^r$ using Proposition
\ref{prop:reg.sol}, and so obtain the regularized field
\begin{equation*}
\RFr =	\tRFr =\sum_{\ell=0}^{\trdeg}\sum_{m=-\ell}^{\ell}\coeshr\shY,
\end{equation*}
again with the use of the HEALPy package.

%
\begin{figure}[ht]
\centering
\scalebox{0.5}{
\includegraphics[trim = 0mm 0mm 0mm 0mm, width=\textwidth]{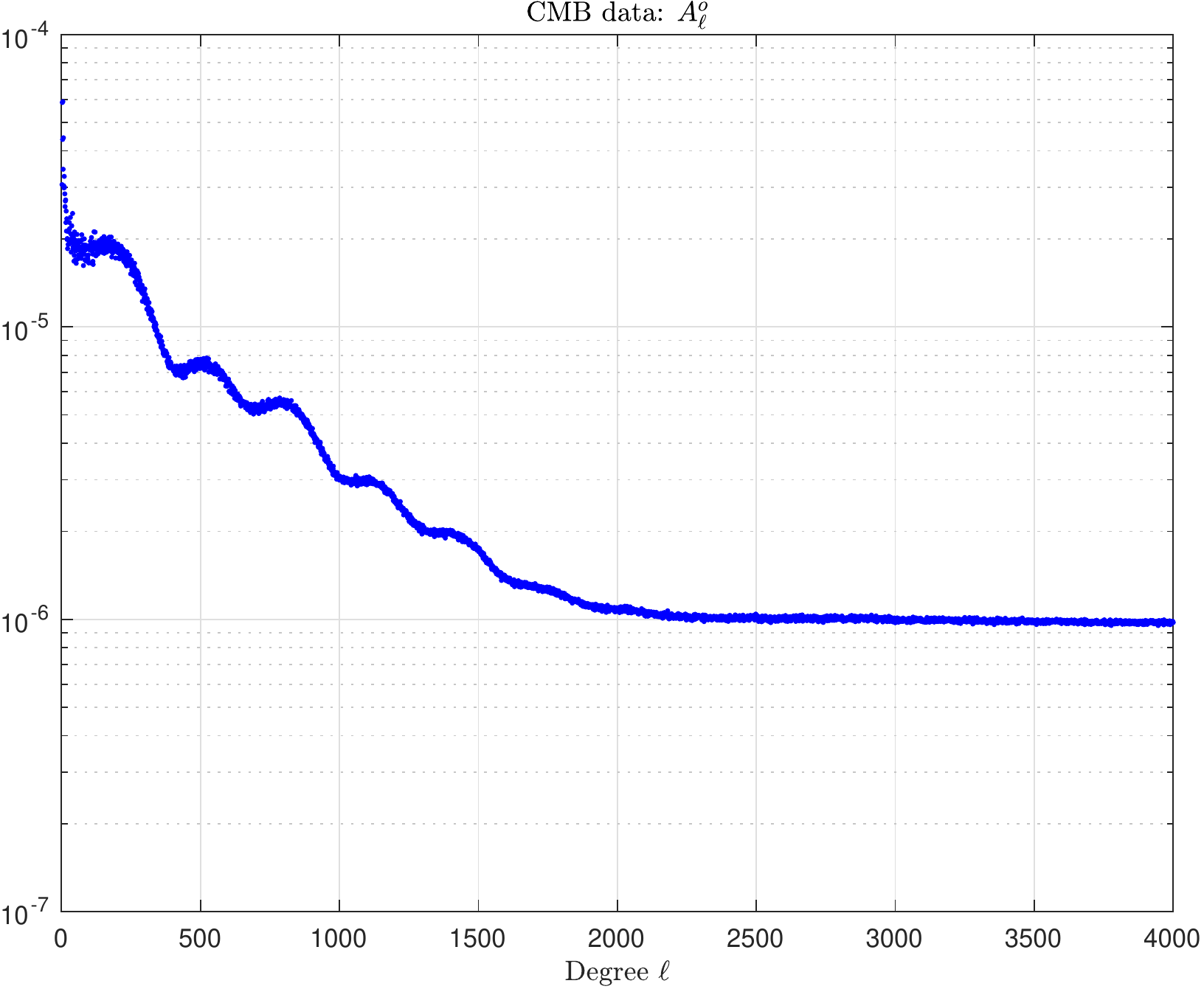}
}
\caption{The observed field $\Alo$.}\label{fig:CMB_Alo}
\end{figure}
\subsection{Choosing the degree scaling parameters $\beta_\ell$}\label{sec:beta}
The degree scaling parameters $\beta_{\ell}$ can be chosen to reflect the decay of the angular
power spectrum of the observed data.
For the CMB data in Figure~\ref{fig:CMB_Alo} there is
remarkably little decay in $\Alo$ for degrees $\ell$
between $2,000$ and $4,000$,
so we choose $\beta_{\ell}=1$ for $\ell=0,\dots,L$.
Note that, if the true data correspond to a field
that is not band-limited but has finite $L_2(\sph{2})$ norm,
then $\Alo$ must eventually decay, and decaying
$\beta_{\ell}$ would then be
appropriate for $\ell>L$.

\subsection{Choosing the regularization parameter $\lambda$}\label{sec:eff}
Now we turn to the choice of the regularization parameter $\lambda$.
We recall from Proposition~\ref{prop:reg.sol} that $\Alr/\Alo$ depends directly on
the ratio $\Alo/(\lambda \beta_\ell)$, and that $\Alr = 0$ if the latter ratio is $\le 1$,
or, since we have chosen $\beta_\ell=1$, if $\Alo \le \lambda$.
It is therefore very clear from Figure~\ref{fig:CMB_Alo} that the sparsity
(i.e. the percentage of the coefficients $\coeshr$ that are zero) will depend
sensitively on the choice of $\lambda$. In Figure~\ref{fig:CMBpowspec} we
illustrate the effect of two choices of $\lambda$ on the computed values of $\Alr$.
In the left panel within Figure~\ref{fig:CMBpowspec} the choice is
$\lambda = 1.05\times 10^{-6}$,
while in the right panel the value of $\lambda$ is $9.75\times 10^{-7}$, about $7\%$ smaller.
In the right panel, the sparsity is less than $10\%$,
whereas on the left it is $72.1\%$.  This means that of the original coefficients (more than $16$ million of them)
 only $4.5$ million are now non-zero.

\begin{figure}[ht]
 \centering
  \begin{minipage}{\textwidth}
  \centering
\begin{minipage}{0.485\textwidth}
\centering
  \includegraphics[trim = 0mm 0mm 0mm 0mm, width=\textwidth]{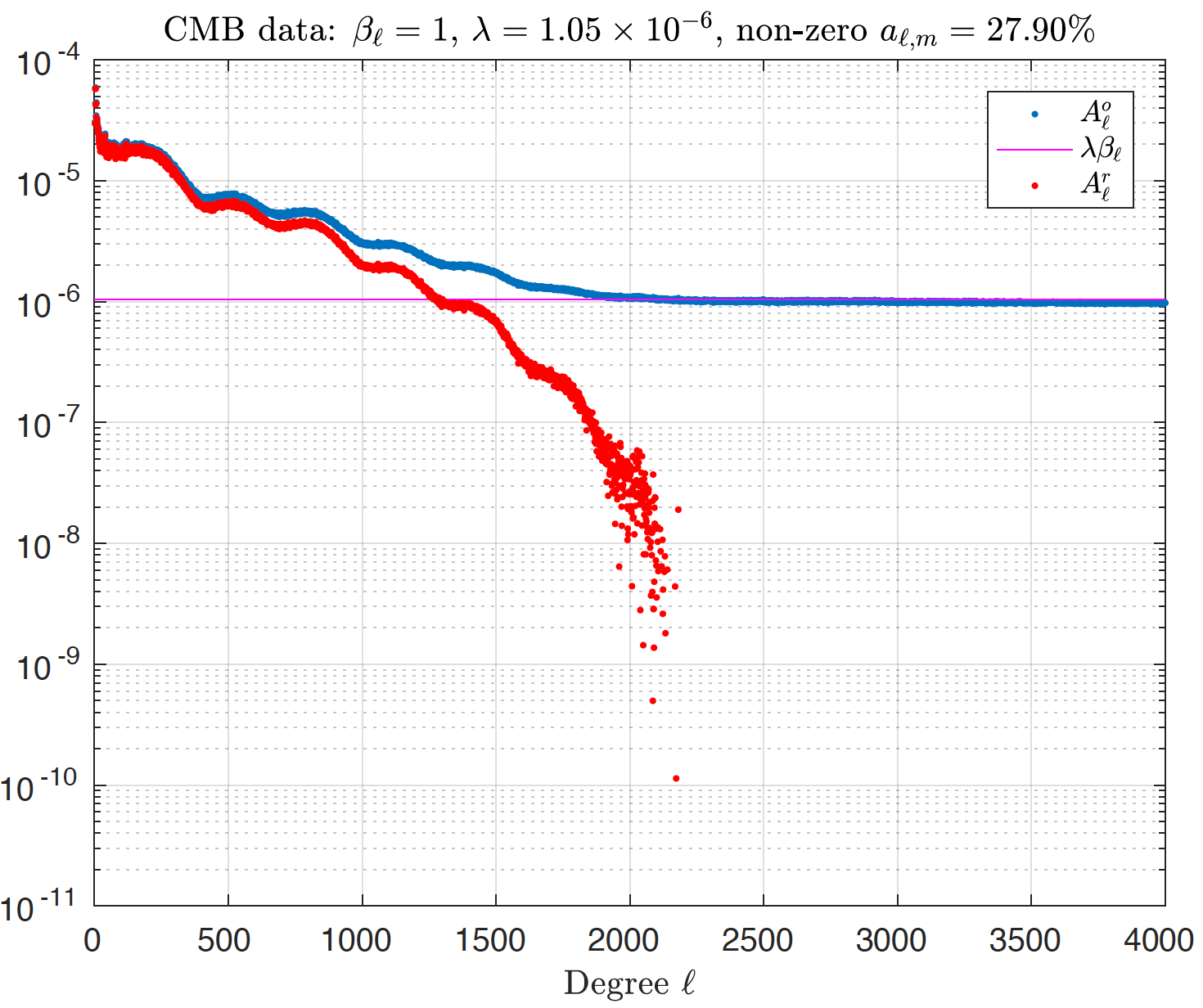}\\
\end{minipage}
\hspace{0.1cm}
\begin{minipage}{0.485\textwidth}
\centering
  \includegraphics[trim = 0mm 0mm 0mm 0mm, width=\textwidth]{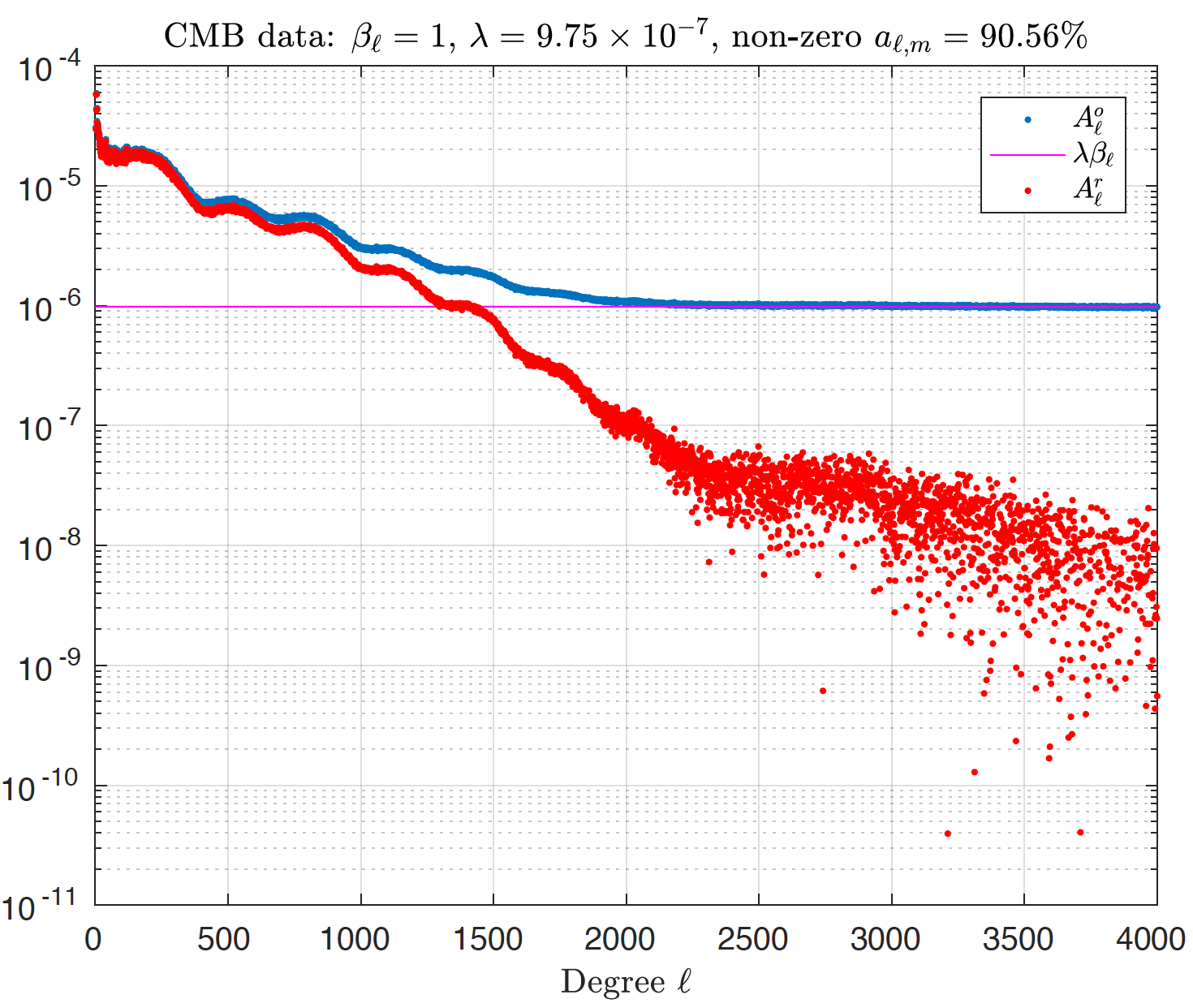}\\
\end{minipage}
  \begin{minipage}{0.9\textwidth}
  \caption{The regularized field $\Alr$ with $\beta_{\ell}=1$, and with
$\lambda=1.05\times10^{-6}$ (left graph)
and $\lambda=9.75\times10^{-7}$ (right graph).}\label{fig:CMBpowspec}
  \end{minipage}
\end{minipage}
\end{figure}

\subsection{Efficient frontier}
A more systematic approach to choosing the regularization
parameter $\lambda$ is to
make use of the Pareto efficient
frontier \cite{OsbPT2000a,DauFL2008,vdBerFri2009}.
The efficient frontier of the multi-objective problem with two objectives,
$\| \vcoesh \|_{1,2}$ and $\| \vcoesh - \vcoesho\|_2^2$,
is the graph obtained by plotting the optimal values
of these two quantities on the $y$ and $x$ axes respectively
as $\lambda$ varies.
As illustrated in the left figure in Figure~\ref{fig:CMBpareto}
for the CMB data,
the graph of the efficient frontier is in this case a continuous
piecewise quadratic,
with knots when the number of degrees $\ell$ with $\Alo/\beta_\ell > \lambda$
changes, that is when the degree set $\Gamma(\lambda)$ changes. In the
figure, $\lambda$ is increasing from left (when $\lambda=0$) to right
where $\vcoeshr$ vanishes at $\lambda= 5.89\times10^{-5}$. The point on
the graph when $\lambda=1.48 \times 10^{-5}$
is shown in the figure. At this value of $\lambda$ the discrepancy
$\|\vcoeshr-\vcoesho\|_2^2$ has the value $10^{-7}$, while
$\|\vcoeshr \|_{1,2}= \sum_{\ell=0}^\infty \beta_\ell A_\ell$ has the value $\kappa=1.21 \times 10^{-3}$.
\begin{figure}[ht]
  \centering
  \begin{minipage}{\textwidth}
  \centering	
  \begin{minipage}{\textwidth}
  \centering
  \begin{minipage}{0.485\textwidth}
  \includegraphics[trim = 0mm 1mm 0mm 6mm, width=1\textwidth]{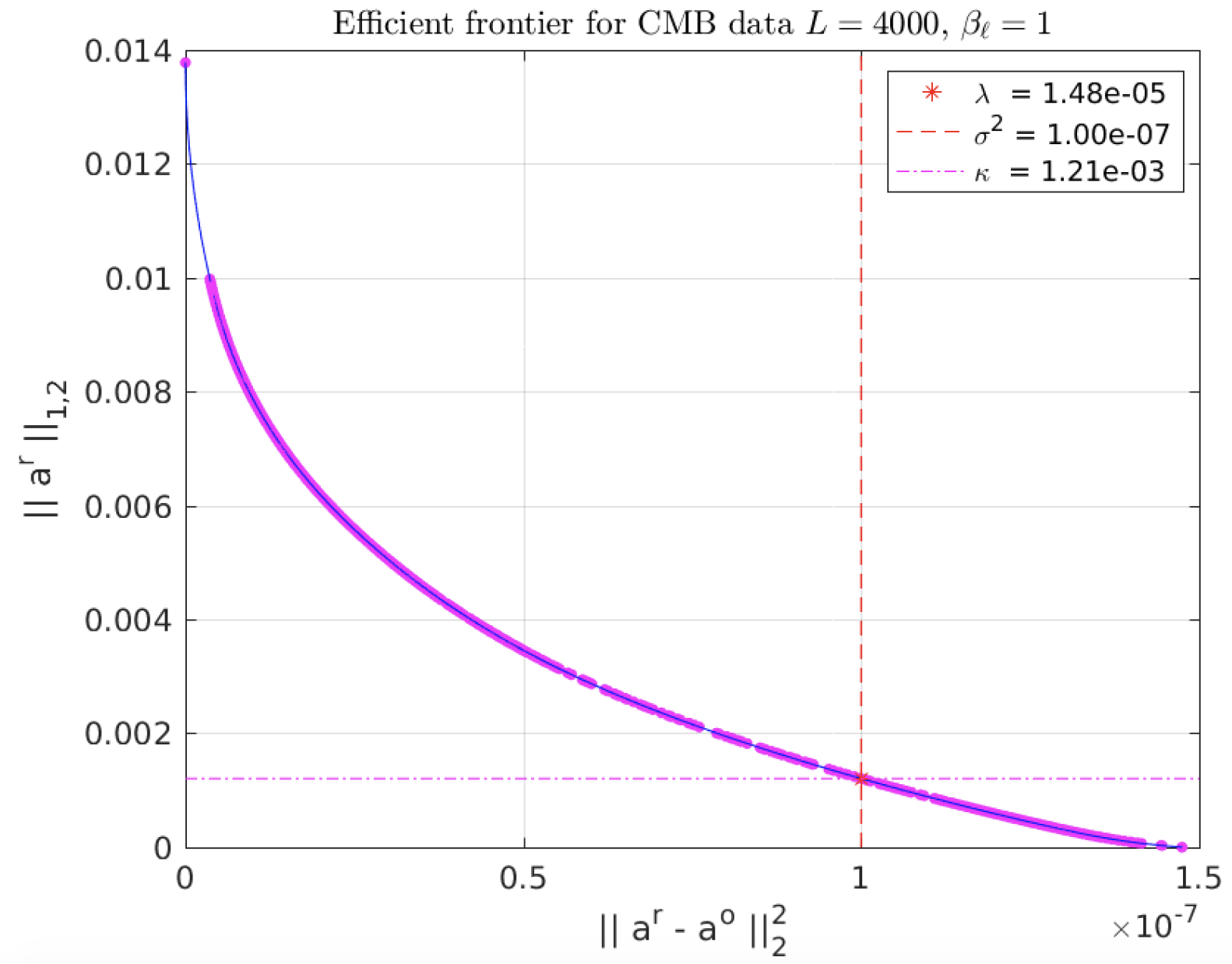}
  \end{minipage}
  \hspace{0.1cm}
  \begin{minipage}{0.485\textwidth}
  \includegraphics[width=\textwidth]{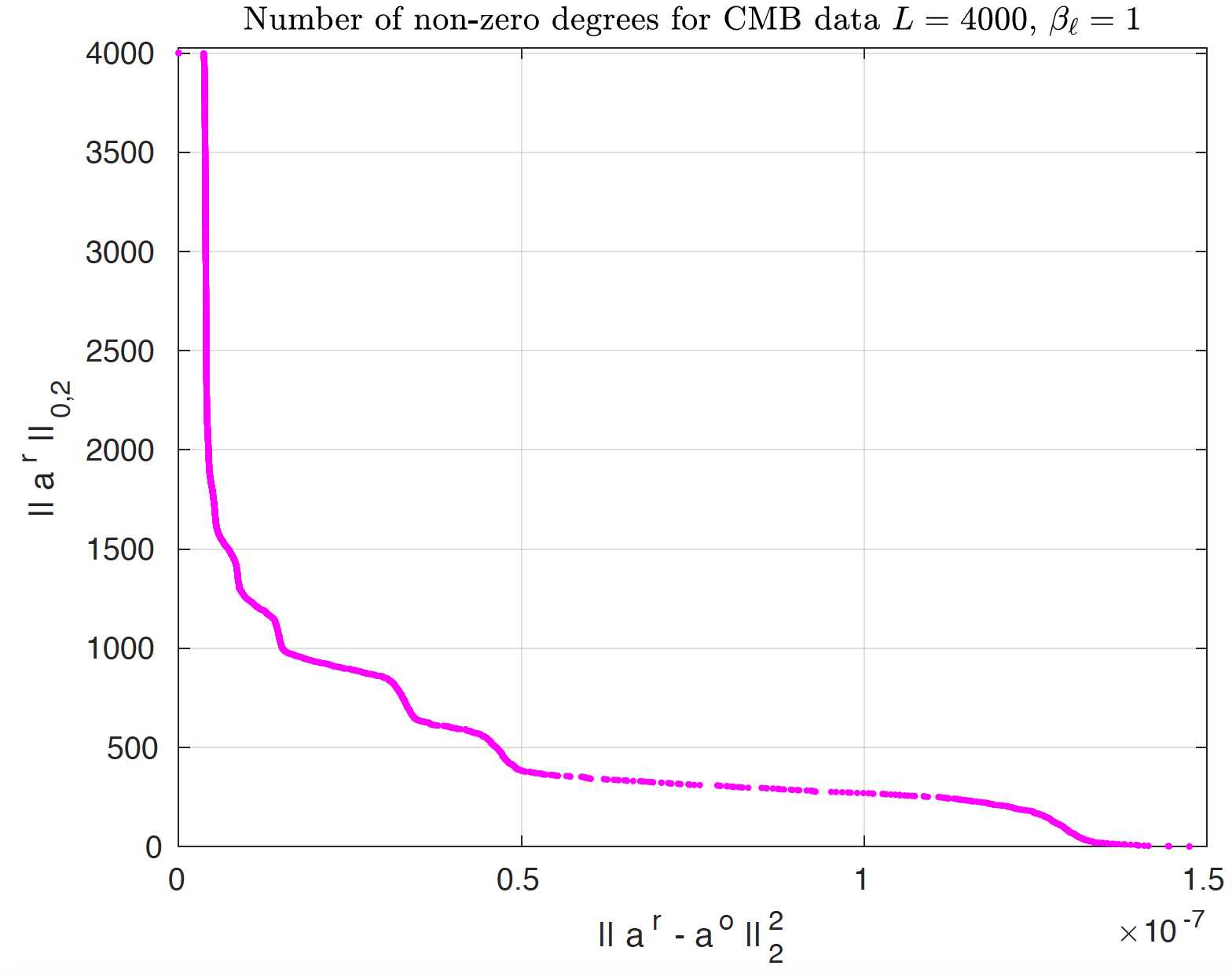}
  \end{minipage}\vspace{-2mm}
  \end{minipage}\\[2mm]
  \begin{minipage}{0.9\textwidth}
  \caption{Efficient frontiers of $\norm{\vcoeshr}{1,2}$ (left) and $\norm{\vcoeshr}{0,2,L}$ (right) against discrepancy $\norm{\vcoeshr - \vcoesho}{2}^2$ on the CMB data for $\beta_{\ell}=1$ and $L=4,000$.}\label{fig:CMBpareto}
  \end{minipage}
  \end{minipage}
\end{figure}

The idea of the efficient frontier is that each point on the frontier
corresponds to an optimal solution for some $\lambda$,
while points above the frontier are feasible but not optimal. At points on the frontier,
one objective can be improved only at the expense of making the other worse.
The appendix shows how to determine the corresponding value of
the regularization
parameter $\lambda$ given either $\sigma$ or $\kappa$ for models (\ref{eq:model2})
or (\ref{eq:model3}).
One can specify the value of $\lambda$ or the discrete discrepancy $\norm{\vcoeshr -\vcoesho}{2}^{2}$ (equivalent to specifying $\sigma$ in \eqref{eq:model2}) or the norm $\norm{\vcoeshr}{1,2}$ (equivalent to specifying $\kappa$ in \eqref{eq:model3}).

In the right figure in Figure~\ref{fig:CMBpareto}, we plot the $\ell_{0}$-norm
defined by
\begin{equation*}
  \norm{\vcoeshr}{0,2,L} := \sum_{\ell=0}^L 1_{\{ {\Alr>0}\}} =\# \Gamma(\lambda),
\end{equation*}
so $\| \vcoeshr \|_{0,2,L}$ counts the number of degrees $\ell=0,\ldots,L$
with at least one non-zero coefficient,
against $\norm{\vcoeshr - \vcoesho}{2}^2$ to more directly compare sparsity and data fitting.
This is a piecewise constant graph
with discontinuities at the values of $\lambda$ when the degree set $\Gamma(\lambda)$ changes.
From this graph it is clear that high sparsity (or small $\ell_0$ norm) implies large discrepancy of the regularized field.
\subsection{Scaling to preserve the $L_2$ norm}
\begin{figure}[ht]
 \centering
  \begin{minipage}{\textwidth}
  \centering
\begin{minipage}{0.485\textwidth}
\centering
 \includegraphics[trim = 0mm 0mm 0mm 0mm, width=\textwidth]{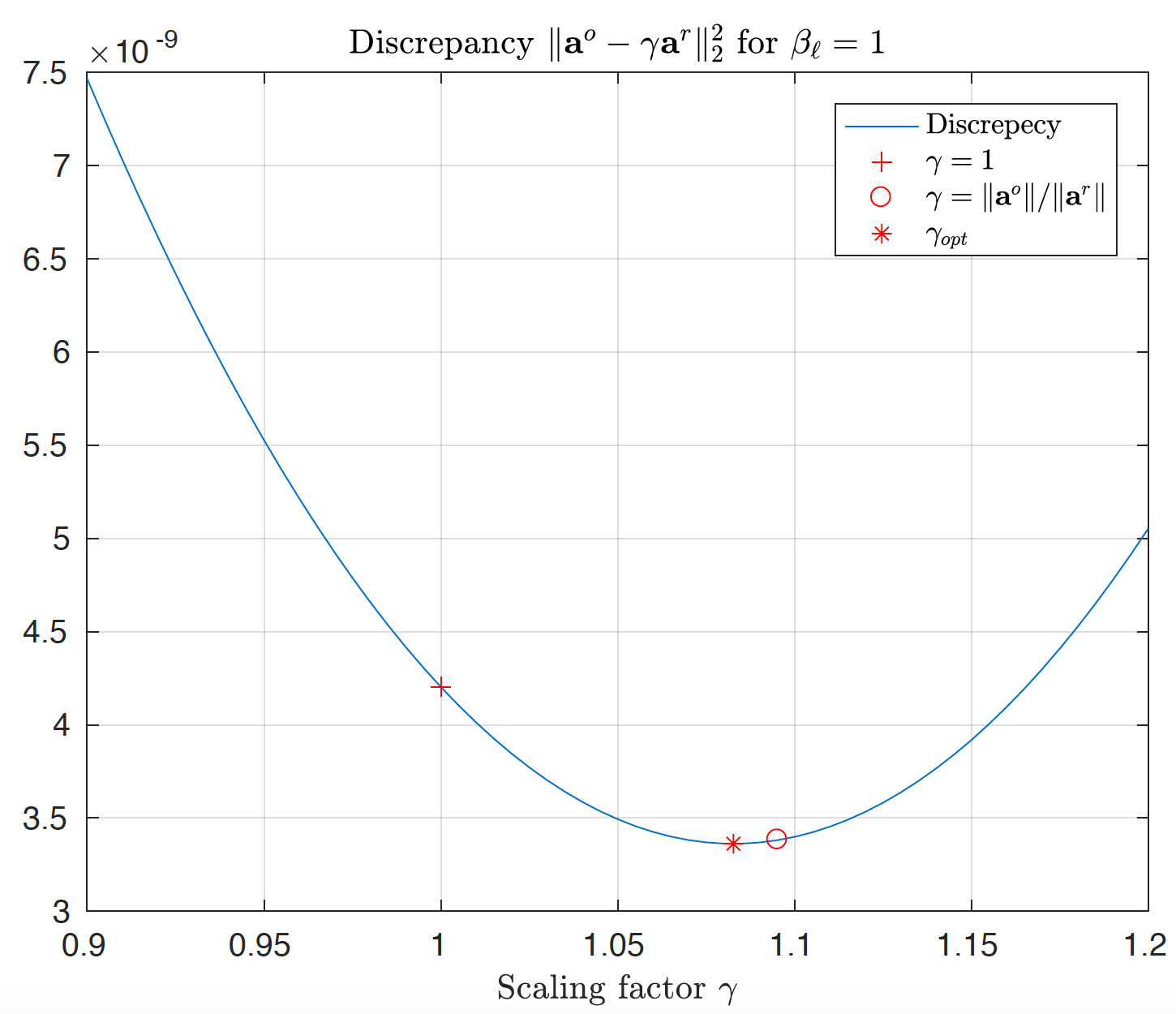}\\[1mm]
\end{minipage}
\begin{minipage}{0.01\textwidth}
\end{minipage}
  \begin{minipage}{0.9\textwidth}
  \caption{The (quadratic) curve of discrepancy $\norm{\vcoesho-\gamma\vcoeshr}{2}^{2}$ with respect to norm-scaling factor $\gamma$ for $\beta_{\ell}=1$ for
the CMB data with $\lambda=1.05\times10^{-6}$.
The circle corresponds to the value of $\gamma$ for which
$\|\gamma \vcoeshr \|_2 = \| \vcoesho\|_2$. The star corresponds
to the value $\gamma_{\rm opt}$ for which the discrepancy is minimized.}
\label{fig:CMBdiscrep}
  \end{minipage}
\end{minipage}
\end{figure}

The scaling factor $\gamma$ can be chosen as in \eqref{eq:scalf} so that the $L_{2}$ norms of the observed data and regularized solution are equal.

Figure~\ref{fig:CMBdiscrep} illustrates the relation between the scaling factor
$\gamma$ and the discrepancy $\norm{\vcoesho-\gamma\vcoeshr}{2}^{2}$ for the
CMB data
with $\lambda=1.05\times10^{-6}$, corresponding to the left panel in Figure~\ref{fig:CMBpowspec}.
The choice of $\gamma$ in \eqref{eq:scalf} that equates the $L_2$-norms $\|\vcoesho\|_2$ and $\|\gamma \vcoeshr\|_2$ makes the discrepancy $\norm{\vcoesho-\gamma\vcoeshr}{2}^{2}$ close to the optimal choice
in the sense of minimizing the discrepancy.
It also shows that $\gamma=1$ (no scaling) 
gives a much larger discrepancy.

\subsection{Errors and sparsity for the regularized CMB field}\label{sec:CMB}
Table~\ref{tab:no.l2nrm.coeff.CMB} gives errors and sparsity results for the regularized CMB field. Included for comparison is the
Fourier reconstruction of degree $L=4,000$ (with $(L+1)^{2}$ coefficients $\coesh$, $m=-\ell,\dots,\ell$, $\ell=0,\dots,L$),
for which the errors should be zero in the absence of rounding errors. For the regularized field the computations use $\beta_{\ell}=1$
for $\ell=0,\dots,L$, and two values of the regularization parameter, namely
 $\lambda=1.056\times10^{-6}$ and $\lambda=9.75\times10^{-7}$
as used in Figure~\ref{fig:CMBpowspec}, and the errors are given for both the unscaled case (i.e. with $\gamma=1$) and scaled with $\gamma$
chosen as in \eqref{eq:scalf} to equate the $L_{2}$-norms of the observed and
regularized fields. The sparsity is the percentage of the regularized coefficients $\coeshr$
which are zero. The $L_{2}$ errors are estimated
by equal weight quadrature at the HEALPix points, while the $L_{\infty}$ errors are estimated by the maximal absolute error at the HEALPix points.


\begin{table}[htb]
\centering
\begin{minipage}{0.98\textwidth}
\centering
\scriptsize
\begin{tabular}{l*{12}{c}c}
\toprule
     &   \multirow{2}{*}{Fourier} & \multicolumn{2}{c}{$\lambda=1.05e{-6}$} & &\multicolumn{2}{c}{$\lambda=9.75e{-7}$} \\ \cline{3-4}\cline{6-7}
     &                         & Unscaled regularized & Scaled regularized  & & Unscaled regularized & Scaled regularized \\
\midrule
Sparsity  &   $0\%$   & $72.1\%$   &    $72.1\%$    &&    $9.44\%$ & $9.44\%$ \\
Scaling $\scalf$    &   - &   $1$   & $1.0953$  &&    $1$   &   $1.0888$  \\
 $L_2$ errors  &   $8.02e{-12}$   &   $6.48e{-05}$   &   $5.82e{-05}$
    &&   $6.16e{-05}$  & $5.54e{-05}$  \\
 $L_{\infty}$ errors  &   $6.34e{-11}$   &   $1.52e{-03}$   &  $1.50e{-03}$    &&   $1.47e{-03}$   & $1.44e{-03}$  \\
\bottomrule
\end{tabular}
\end{minipage}
\begin{minipage}{0.9\textwidth}
\vspace{3mm}
\caption{Sparsity and estimated $L_2$ and $L_{\infty}$ errors for the regularized fields, both scaled and unscaled,
from the CMB data using $\beta_{\ell}=1$, degree $L=4,000$, and two values of $\lambda$.}
\label{tab:no.l2nrm.coeff.CMB}
\end{minipage}
\end{table}

  Figures~\ref{fig:CMB.scale.beta1_lam1.05e-06} and \ref{fig:CMB.scale.beta1_lam1.05e-06.err} show respectively
  the realization of the scaled regularized field and its pointwise errors 
  with $\beta_{\ell}=1$, $\lambda=1.05\times10^{-6}$ and $\gamma\approx1.0953$, the first parameter
  choice in Table \ref{tab:no.l2nrm.coeff.CMB}. This regularized field uses only
  $27.90\%$ of the coefficients in the Fourier approximation.
  Figures~\ref{fig:CMB.scale.beta1_lam9.75e-07} and \ref{fig:CMB.scale.beta1_lam9.75e-07.err} show the realization of the scaled regularized field and its errors for the second parameter choice in Table \ref{tab:no.l2nrm.coeff.CMB}, which uses $90.56\%$ of the coefficients.

\begin{figure}[ht]
  \centering
  \begin{minipage}{\textwidth}
  \centering
  \begin{minipage}{\textwidth}
  \begin{minipage}{0.485\textwidth}
  \centering
  \includegraphics[trim = 0mm 0mm 0mm 0mm, width=1.03\textwidth]{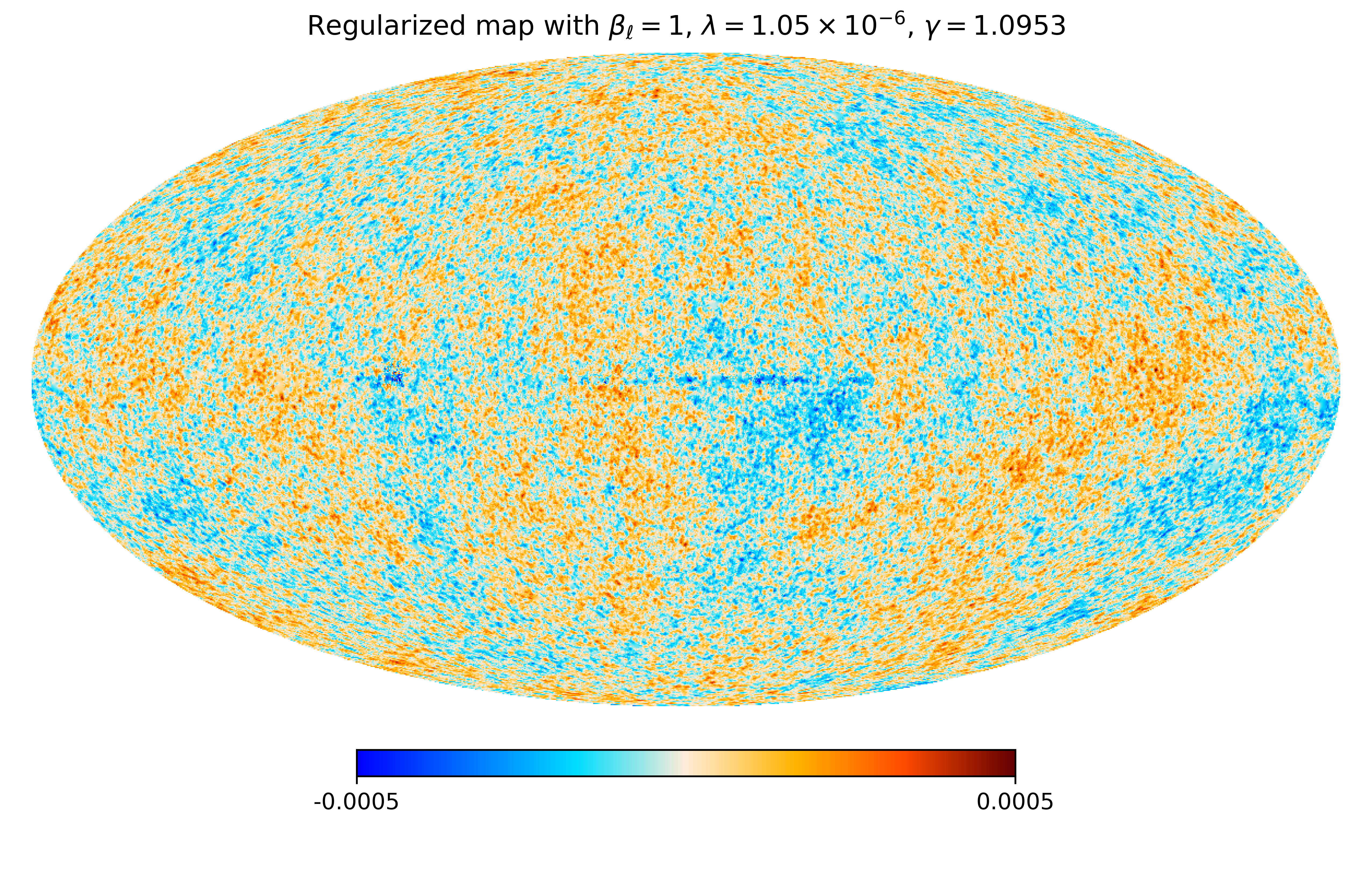}\\[-4mm]
  \subcaption{Scaled regularized field, $\lambda=1.05\times10^{-6}$, $\gamma=1.0953$}\label{fig:CMB.scale.beta1_lam1.05e-06}
  \end{minipage}
  \hspace{0.0\textwidth}
  \begin{minipage}{0.485\textwidth}
  \centering
  \includegraphics[trim = 0mm 0mm 0mm 0mm, width=1.03\textwidth]{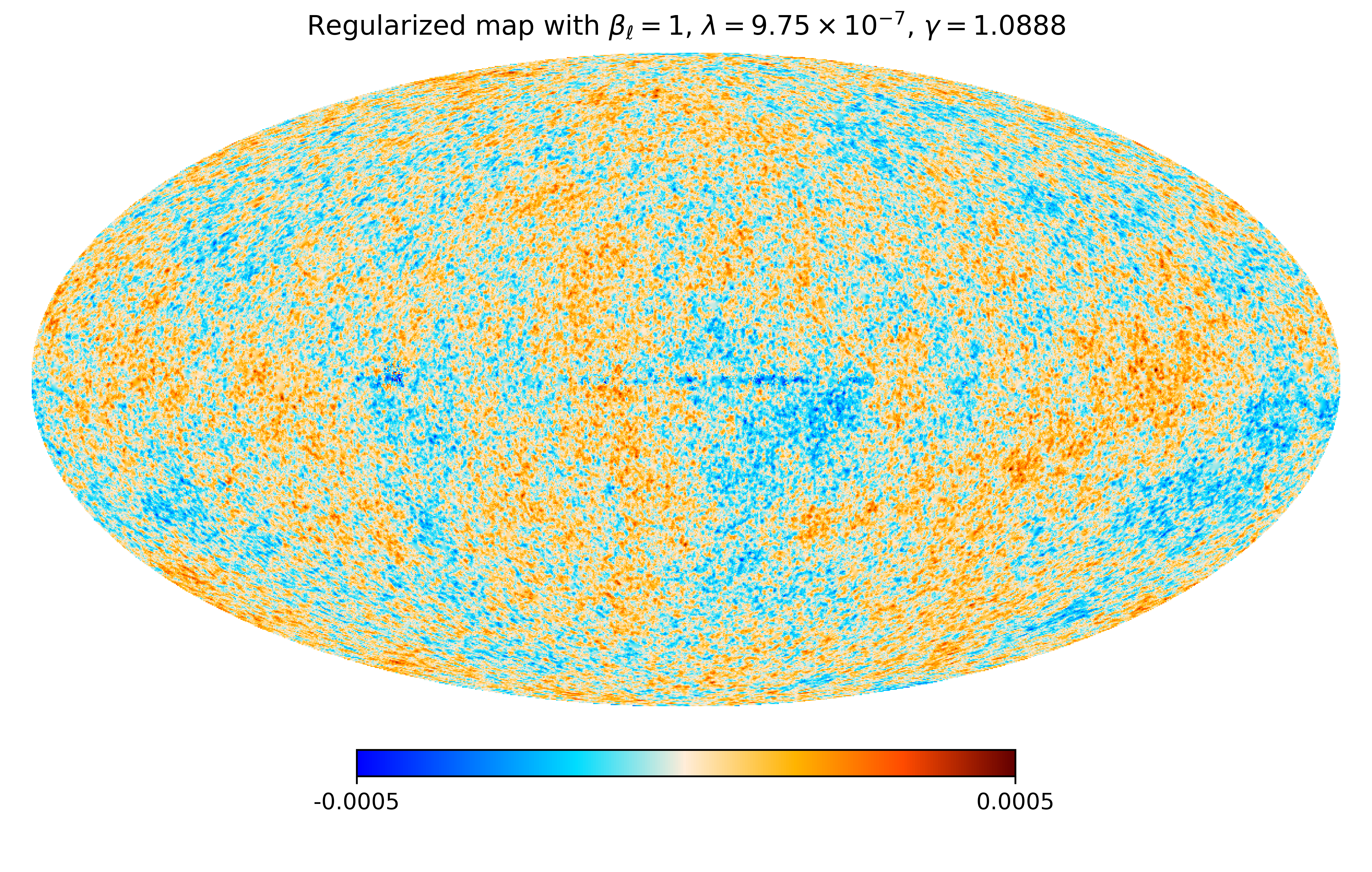}\\[-4mm]
  \subcaption{Scaled regularized field, $\lambda=9.75\times10^{-7}$, $\gamma=1.0888$}\label{fig:CMB.scale.beta1_lam9.75e-07}
  \end{minipage}
  \end{minipage}\\[4mm]
   \begin{minipage}{\textwidth}
  \begin{minipage}{0.485\textwidth}
  \centering
 \includegraphics[trim = 0mm 0mm 0mm 0mm, width=1.03\textwidth]{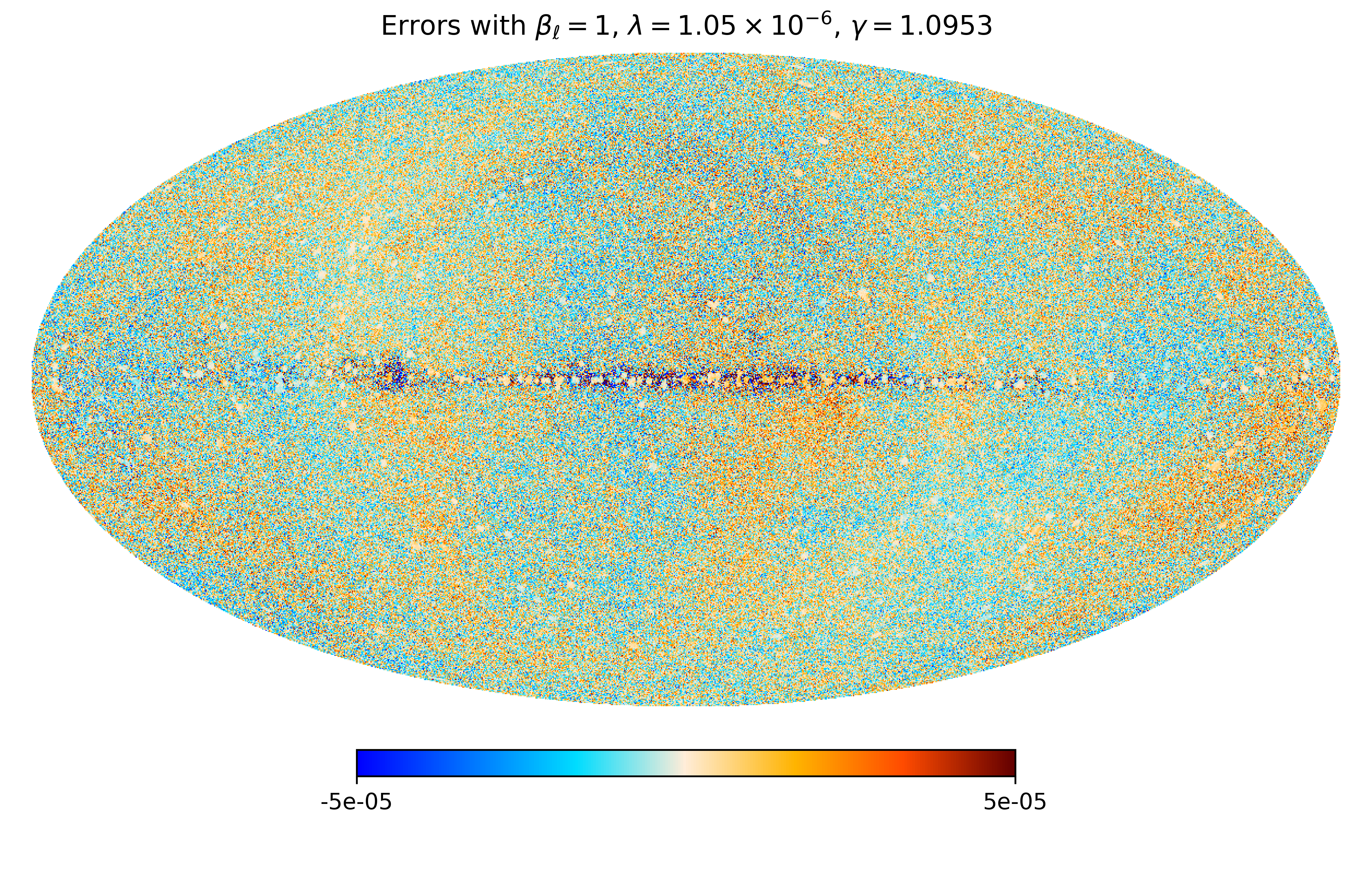}\\[-4mm]
\subcaption{Errors of (a)}\label{fig:CMB.scale.beta1_lam1.05e-06.err}
  \end{minipage}
  \hspace{0.0\textwidth}
  \begin{minipage}{0.485\textwidth}
  \centering
 \includegraphics[trim = 0mm 0mm 0mm 0mm, width=1.03\textwidth]{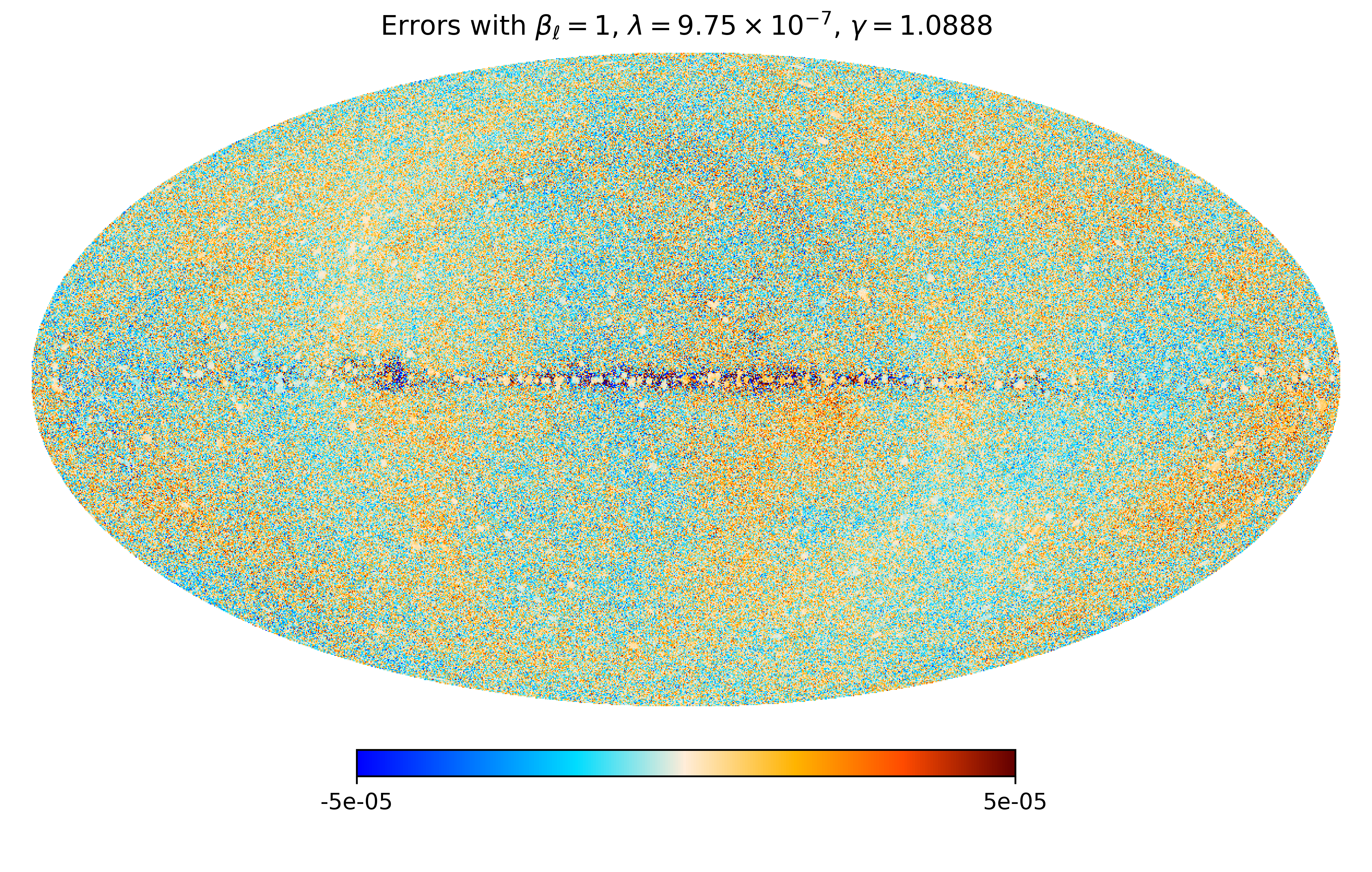}\\[-4mm]
  \subcaption{Errors of (b)}\label{fig:CMB.scale.beta1_lam9.75e-07.err}
  \end{minipage}
  \end{minipage}
  \end{minipage}\\
  \begin{minipage}{0.9\textwidth}
\vspace{3mm}
  \caption{\scriptsize
  (a) and (b) show the realizations of the scaled regularized random field for $\beta_{\ell}=1$
  with $\lambda=1.05\times10^{-6}$ and $\lambda=9.75\times10^{-7}$ from the original CMB field, truncation degree $L=4,000$;
(c) and (d) show the pointwise errors of (a) and (b) respectively, with the range of the color map one tenth of that in (a) and (b).}\
\label{fig:CMB.reg.beta1}
\end{minipage}
\end{figure}

The errors in Figure \ref{fig:CMB.reg.beta1} should be
considered in relation to the
$\Lp{2}{2}$ and $\Lp{\infty}{2}$ norms of the original CMB field,
which are $3.84e{-04}$ and $1.86e{-03}$ respectively.  (The latter number implies that there are points of the original map corresponding to
Figure \ref{fig:CMB_Original} with values that exceed the limits of the color map by a factor of nearly 4. However, points exceeding the limits of the color map are relatively rare.)
We can observe that the magnitudes of the pointwise errors in Figures~\ref{fig:CMB.reg.beta1}
are mostly an order of magnitude smaller than
the magnitude of the corresponding fields. The largest errors occur near the equator where the original CMB map was masked and then inpainted using other parts of the data, see \cite{Planck2016IX}. Outside the region near
the equator the errors in Figures~\ref{fig:CMB.scale.beta1_lam1.05e-06.err} and \ref{fig:CMB.scale.beta1_lam9.75e-07.err} vary from place to place but on
the whole are uniformly distributed.

Table~\ref{tab:no.l2nrm.coeff.CMB} and Figure~\ref{fig:CMB.reg.beta1} show that our appropriate choice of the regularization parameter $\lambda$
can make the errors of the scaled regularized field sufficiently small.
Moreover, the larger of the two choices of $\lambda$
significantly increases the sparsity while only slightly increasing the approximation error.


\section*{Acknowledgements}
We are grateful for helpful discussions with Professor Marinucci regarding strong isotropy of random fields on spheres.
We are also grateful for the use of data from the Planck/ESA mission, downloaded from the Planck Legacy Archive. This research includes extensive computations using the Linux computational cluster Katana supported by the Faculty of Science, The University of New South Wales, Sydney. We acknowledge support
from the Australian Research Council under Discovery Project DP180100506.
The last author was supported under the Australian Research Council's 
Discovery Project DP160101366.


\section*{References}

\bibliographystyle{abbrv}
\bibliography{sparse_r}


\appendix
\section{Relation to constrained models}
Consider, for simplicity, the case when $\RF$ is band-limited with maximum degree $L$.
When their constraints are active, the two constrained
models (\ref{eq:model2}) and (\ref{eq:model3}) are equivalent
to the regularized model (\ref{eq:model1}).
This equivalence is detailed below, where we also show how to calculate
the value of the regularization parameter corresponding to
an active constraint, see (\ref{eq:siglam}) and (\ref{eq:kaplam}) below.

Consider the optimization problem (\ref{eq:model2}) with the
data fitting constraint $\| \vcoesh - \vcoesho \|_2^2 \leq \sigma^2$.
Introducing a Lagrange multiplier $\mu\in\mathbb{R}$ for the constraint,
the optimality conditions are, using (\ref{eq:normderiv}),
\begin{equation}
  \begin{array}{ll}
 \beta_\ell A_\ell^{-1} \coesh  + 2 \mu (a_{\ell, m} - \coesho) = 0, \quad
     m = -\ell,\ldots,\ell & \mbox{when } A_\ell > 0, \\
     \coesh = 0, \quad m = -\ell,\ldots,\ell & \mbox{when } A_\ell = 0, \\ 
    \norm{\vcoesh - \vcoesho}{2}^2 - \sigma^2 \leq 0,
     & \mbox{primal feasibility}, \\
    \mu \geq 0, & \mbox{dual feasibility}, \\ 
    \mu \left(\norm{\vcoesh - \vcoesho}{2}^2 - \sigma^2\right) = 0,
     & \mbox{complementarity}. 
     \end{array}
     \label{eq:opt2}
\end{equation}
If $\sigma^2 \geq \| \vcoesho \|_2^2$, the unique solution is $\vcoesh = \mathbf{0}$,
that is $\coesh = 0, \; m = -\ell,\ldots,\ell, \; \ell=0,\dots,L$.
On the other hand if $\sigma = 0$, the unique solution is $\vcoesh = \vcoesho$,
so $\coesh = \coesho, \; m = -\ell,\ldots,\ell, \; \ell=0,\dots,L$.

Comparing the optimality conditions (\ref{eq:opt1a}) and the first equation of (\ref{eq:opt2})
shows that, for $\mu > 0$,
\[
   \lambda = \frac{1}{2\mu}.
\]
In terms of $\mu \geq 0$ we define the degree sets, similarly to
 \eqref{eq:Gamma}, by
\begin{equation*}
  \widetilde{\Gamma}(\mu) :=
\left\{0 \leq \ell \leq L: \frac{\beta_\ell}{2 \Alo} < \mu\right\}
=\Gamma(\lambda),\qquad
  \widetilde{\Gamma}^{c}(\mu) := \left\{0\leq\ell\leq L: \frac{\beta_\ell}{2 \Alo} \geq \mu\right\} = \Gamma^{c}(\lambda).
\end{equation*}
For $\mu = 0$, $\widetilde{\Gamma}(0) = \emptyset$,
while for $\mu > \max_{0\le \ell\le L} \frac{\beta_\ell}{2 \Alo}$,
the index set $\widetilde{\Gamma}^{c}(\mu) = \emptyset$.
The optimality condition (\ref{eq:opt2}) gives,
for all $\ell\in\widetilde{\Gamma}(\mu)$ with $A_\ell > 0$,
\begin{equation*}
  \coesh = \alr \coesho, \quad m = -\ell,\ldots,\ell,
  \quad \mbox{where} \;\;
  \alr = \frac{A_\ell}{A_\ell + \beta_\ell/(2\mu)}.
\end{equation*}
As before, $0 < \alr < 1$, $A_\ell =  \alr  \Alo$, so
\begin{equation*}
  A_{\ell} = \left\{ \begin{array}{cl}
  \displaystyle \Alo - \beta_\ell/(2\mu) & \quad\mbox{for }
  \ell\in \widetilde{\Gamma}(\mu), \\[1ex]
  0 & \quad\mbox{for } \ell\in \widetilde{\Gamma}^{c}(\mu).
  \end{array}\right.
\end{equation*}
Given that the constraint is active,
the value of $\lambda$ corresponding to $\sigma$ can be found by solving, see (\ref{eq:lamssq}),
\begin{equation}\label{eq:siglam}
      \lambda^2 \sum_{\ell\in\Gamma(\lambda)} \beta_\ell^2 =
   \sigma^2 - \sum_{\ell\in\Gamma^{c}(\lambda)} (\Alo)^2 .
\end{equation}
The only issue here is finding the sets $\Gamma(\lambda)$ and $\Gamma^c(\lambda)$ when we start from the model \eqref{eq:model2}.
As they only change when $\lambda$ is $\frac{\Alo}{\beta_\ell}$,
this can be done by sorting the values $\frac{\Alo}{\beta_\ell}$,
and finding the largest value of $\lambda' = \frac{\Alo[\ell']}{\beta_{\ell'}}$
such that $\| \vcoesh - \vcoesho \|_2^2 \leq \sigma^2$.
And then solving (\ref{eq:siglam}) using $\Gamma(\lambda')$ and $\Gamma^c(\lambda')$.

Consider now the LASSO type model (\ref{eq:model3}) with a constraint
 $\| \vcoesh \|_{1,2}\leq \kappa$.
Introducing a Lagrange multiplier $\nu\in\mathbb{R}$ for the constraint,
the optimality conditions are, again using (\ref{eq:normderiv}),
\begin{equation}
  \begin{array}{ll}
   (\coesh - \coesho) + \nu \beta_\ell  A_\ell^{-1} \coesh  = 0, \quad
     m = -\ell,\ldots,\ell & \mbox{when } A_\ell > 0, \\
     \coesh = 0, \quad m = -\ell,\ldots,\ell & \mbox{when } A_\ell = 0, \\ 
    \norm{\vcoesh}{1,2} - \kappa \leq 0,
     & \mbox{primal feasibility}, \\
    \nu \geq 0, & \mbox{dual feasibility}, \\ 
    \nu \left( \norm{\vcoesh}{1,2} - \kappa \right) = 0,
     & \mbox{complementarity}. 
     \end{array}
     \label{eq:opt3}
\end{equation}
If $\kappa \geq \norm{\vcoesh}{1,2}$ then the solution is $\vcoesh = \vcoesho$
with $\nu =0$.
If $\kappa = 0$ then the solution is $\vcoesh = \mathbf{0}$.
The first equation in (\ref{eq:opt3}), for $A_\ell > 0$, gives
\begin{equation*}
  \coesh = \alr \coesho, \quad m = -\ell,\ldots,\ell,
  \quad \mbox{where} \;\;
  \alr = \frac{1}{1 + \nu \beta_\ell A_\ell^{-1}} = \frac{A_\ell}{A_\ell + \nu \beta_\ell}.
\end{equation*}
When the constraint is active and $\nu > 0$,
comparing the first equation in (\ref{eq:opt3}) and (\ref{eq:opt1a}),
shows that $\lambda = \nu$.
Given a value for $\kappa$ with $0 < \kappa < \norm{\vcoesho}{1,2}$,
the corresponding value of $\lambda$ satisfies, see (\ref{eq:lamnrm})
\begin{equation}\label{eq:kaplam}
  \lambda \sum_{\ell\in\Gamma(\lambda)} (\beta_\ell)^2 =
  \sum_{\ell\in\Gamma(\lambda)} \beta_\ell \Alo - \kappa.
\end{equation}
Again, the only issue is first determining the set $\Gamma(\lambda)$,
defined in (\ref{eq:Gamma}),
which can be done by sorting the $\frac{\Alo}{\beta_\ell}$
to find the smallest value $\lambda^* = \frac{\Alo[\ell^*]}{\beta_{\ell^*}}$
such that $\| \vcoesh \|_{1,2} \geq \kappa$, and then solving (\ref{eq:kaplam})
using $\Gamma(\lambda^*)$.

\end{document}